\documentstyle[12pt,amsmath,amssymb,amsthm,epsf,graphicx,color]{article}

\def\rth2{{\mathbb{R}}^2}

\newcommand{\Z}{\mathbb{Z}}
\newcommand{\boB}{\mathcal{B}}
\def \R {{\Bbb R}}

\def \esf {{\Bbb S}}
\def \M {{\Bbb M}}
\def \H {{\Bbb H}}
\def \N {{\Bbb N}}
\def \D {{\Bbb D}}
\def \a {\alpha}
\def \be {\beta}
\def \g {\gamma}
\def \G {\Gamma}
\def \ve {\varepsilon}
\newcommand{\dis}{\displaystyle}
\newtheorem{theorem}{Theorem}[section]
\newtheorem{proposition}[theorem]{Proposition}
\newtheorem{definition}[theorem]{Definition}
\newtheorem{lemma}[theorem]{Lemma}
\newtheorem{corollary}[theorem]{Corollary}
\newtheorem{remark}[theorem]{Remark}
\newtheorem{claim}[theorem]{Claim}

\newcommand{\myskip}[1]{}

\DeclareMathOperator{\re}{Re}
\DeclareMathOperator{\im}{Im}

\begin{document}

\title{The Dirichlet problem for the minimal surface equation -with
possible infinite boundary data- over domains in a Riemannian surface}
\author{Laurent Mazet, M. Magdalena Rodr\'\i guez\thanks{Research
    partially supported a CNRS grant and a MEC/FEDER grant no.
    MTM2007-61775. The second author would like to thanks L'Institut
    de Math\'ematiques de Jussieu (UMR 7586) for its hospitality during
    the preparation of this manuscript.} and Harold Rosenberg}

\maketitle

%%%%%%%%%%%%%%%%%%%
%%%%%%%%%%%%%%%%%%%

\section{Introduction}\label{secintro}

%%%%%%%%%%%%%%%%%%%
%%%%%%%%%%%%%%%%%%%

In \cite{jes1}, Jenkins and Serrin considered bounded domains
$D\subset\R^2$, with $\partial D$ composed of straight line segments and
convex arcs. They found necessary and sufficient conditions on the lengths
of the sides of inscribed polygons, which guarantee the existence of a
minimal graph over $D$, taking certain prescribed values (in
$\R\cup\{\pm\infty\}$) on the components of $\partial D$

Perhaps the simplest example is $D$ a triangle and the boundary data is
zero on two sides and $+\infty$ on the third side. The conditions of
Jenkins-Serrin reduce to the triangle inequality  here and the solutions
exists. It was discovered by Scherk in 1835.

This also works on a parallelogram with sides of equal length. One
prescribes $+\infty$ on opposite sides and $-\infty$ on the other two
sides. This solution was also found by Scherk.

The theorem of Jenkins and Serrin also applies to some non-convex domains.
They only require $\partial D$ to be composed of a finite number of
convex arcs, together with their endpoints.

In a very interesting paper \cite{sp3}, Joel Spruck solved the Dirichlet
problem for the constant mean curvature $H$ equation over bounded domains
$D\subset \R^2$, with $\partial D$ composed of circle arcs of curvature
$\pm2H$, together with convex arcs of curvature larger than $2H$. The
boundary data now is $\pm\infty$ on the circle arcs and prescribed
continuous data on the convex arcs. He gave necessary and sufficient
conditions on the perimeter, and area, of inscribed polygons that
solve the Dirichlet problem.

In recent years there has been much activity on this Dirichlet problem
over domains $D$ contained in a Riemannian surface $\M$
\cite{pin1,you}. When $\M$ is the hyperbolic plane $\H^2$, there are
non-compact domains for which this problem has been solved, and interesting
applications have been obtained (see for example \cite{cor2,
hrs1, NeRo2}

In this paper we will extend the solution of this Dirichlet problem to
general domains. In the case of a Riemannian surface $\M$, we consider
non-convex domains (see Section \ref{secth1}). For $\M=\H^2$, we study
non-compact domains.

Our techniques for doing this in $\H^2$ are new (and apply to domains in
arbitrary $\M$). Previously one found a solution to the Dirichlet problem
by taking
limits of monotone sequences of solutions whose boundary data converges to
the prescribed data. A basic tool to make this work is the maximum
principle for solutions: if $u$ and $v$ are solutions and $u\le v$ on
$\partial D$, then $u\le v$ on $D$. However, there are domains for which
the maximum principle fails (we discuss this in Section
\ref{seccontreexemple}). In order to solve the Dirichlet problem in the
absence of a maximum principle we use the idea of divergence lines
introduced by Laurent Mazet in his thesis \cite{mazet0}. This enables us to
obtain
convergent subsequences of non-necessarily monotone sequences.

This lack of a general maximum principle implies that one no longer has
uniqueness (up to an additive constant, in the case of infinite boundary
data) for the solutions. In section \ref{secPrincMax}, we obtain uniqueness
theorems for certain domains and we give examples where this fails.

%%%%%%%%%%%%%%%%%%%
%%%%%%%%%%%%%%%%%%%

\section{Preliminaries}\label{secprel}

%%%%%%%%%%%%%%%%%%%
%%%%%%%%%%%%%%%%%%%

From now on, $\M$ will denote a Riemannian surface. In the following,
$\mbox{div}$, $\nabla$ and $|\cdot|$ are defined with respect to the metric
on $\M$. Let $\Omega$ be a domain in $\M$ and $u:\Omega\to \R$ be a smooth
function. We define $W_u=\sqrt{1+|\nabla u|^2}$. The graph of such a
smooth function $u$ that satisfies
$$
\mbox{div}\left(\frac{\nabla u}{W_u}\right)=0,
$$
is a minimal surface in $\M\times \R$; referred to as a \emph{minimal
graph}. In the following we will denote $X_u=\frac{\nabla u}{W_u}$.

The next results have been proven by Jenkins and Serrin~\cite{jes1}
for $\M=\R^2$, by Nelli and Rosenberg~\cite{NeRo2} when $\M=\H^2$, and
by Pinheiro~\cite{pin1} in the general setting. In fact, these results
were proven for bounded and geodesically convex domains in~\cite{pin1},
although their proofs remain valid in a more general setting.

\begin{theorem}[Compactness theorem]
  \label{th:compactness}
  % p�ginas 20,21,22,23 de Ana Lucia
  Let $\{u_n\}$ be a uniformly bounded sequence of minimal graphs in
  a bounded domain $\Omega\subset\M$. Then, there exists a subsequence
  of $\{u_n\}$ converging on compact subsets of $\Omega$ to a
  minimal graph $u$ on $\Omega$.
\end{theorem}

\begin{theorem}[Monotone convergence theorem]
  \label{th:monotone}
  Let $\{u_n\}$ be an increasing sequence of minimal graphs on a
  domain $\Omega\subset\M$.  There exists an open set ${\cal U}\subset
  \Omega$ (called the {\it convergence set}) such that $\{u_n\}$
  converges uniformly on compact subsets of ${\cal U}$ and diverges
  uniformly to $+\infty$ on compact subsets of ${\cal
    V}=\Omega-\overline{\cal U}$ ({\it divergence set}).  Moreover, if
  $\{u_n\}$ is bounded at a point $p\in \Omega$, then the
  convergence set ${\cal U}$ is non-empty (it contains a neighborhood
  of $p$).
\end{theorem}

Now we recall some results which allow us to describe the divergence set
${\cal V}$ associated to a monotone sequence of minimal graphs.

\begin{lemma}[Straight line lemma]\label{lem:convexhull}
  Let $\Omega\subset\M$ be a domain, $C\subset\partial\Omega$ a convex
  compact arc, and $u\in{\cal C}^0(\Omega\cup C)$ a minimal graph on
$\Omega$.
  Denote by ${\cal C}(C)$ the (open) convex hull of $C$.
  \begin{itemize}
  \item[(i)] If $u$ is bounded above on $C$ and $C$ is strictly convex,
then $u$ is bounded above on $K\cap \Omega$, for every compact
set $K\subset{\cal C}(C)$. 
  \item[(ii)] If $u$ diverges to $+\infty$ or $-\infty$ as we approach
    $C$ within $\Omega$, then $C$ is a geodesic arc.
  \end{itemize}
\end{lemma}

\begin{definition}
Let $u$ be a minimal graph on a domain $\Omega\subset\M$ and assume that
$\partial \Omega$ is arcwise smooth. When $C$ is an arc in $\Omega$ and
$\nu$ is a unit normal to $C$ in $\M$ we define the \emph{flux} of $u$
across $C$ for such choice of $\nu$ by 
$$
F_u(C)=\int_C\langle X_u,\nu\rangle ds,
$$
where $ds$ is the arc length of $C$. Since the vector field $X_u$ is
bounded
and has vanishing divergence, the flux is also defined across a curve
$\Gamma\subset \partial\Omega$, in that case, $\nu$ is chosen to be the
outer normal to $\partial\Omega$.
\end{definition}

In the paper, when a flux is computed across a curve $C$, the curve $C$
will be always seen as part of the boundary of a subdomain. The normal
$\nu$ will then be chosen as the outer normal to the subdomain.

\begin{lemma}\label{lem:flux}
  Let $u$ be a minimal graph on a domain $\Omega\subset\M$.
  \begin{itemize}
  \item[(i)] For every compact bounded domain $\Omega'\subset\Omega$,
    we have $F_u(\partial\Omega')=0$.
  \item[(ii)] Let $C$ be a piecewise smooth interior curve or a convex
    curve in $\partial\Omega$ where $u$ extends continuously and takes
    finite values.  Then $|F_u(C)|<|C|$.
  \item[(iii)] Let $T\subset\partial\Omega$ be a geodesic arc such
    that $u$ diverges to $+\infty$ (resp $-\infty$) as one approaches
    $T$ within $\Omega$. Then $F_u(T)=|T|$ (resp. $F_u(T)=-|T|$).
  \end{itemize}
\end{lemma}

\begin{remark}\label{rem:corner}
  From Lemma~\ref{lem:flux} and the triangle inequality, we deduce
  that, if $u:\Omega\to\R$ is a minimal graph and
  $T_1,T_2\subset\partial\Omega$ are two geodesics where $u$ diverges
  to $+\infty$ as we approach them, then $T_1,T_2$ cannot meet at a
  strictly convex corner (strictly convex with respect to $\Omega$).
\end{remark}

% The last item in the previous lemma is obtained from Assertion 4.9
%  of~\cite{pin1} together with the following lemma.

The last statement in Lemma~\ref{lem:flux} admits the following
generalization.

\begin{lemma}\label{lem:limitflux}
  For each $n\in\N$, let $u_n$ be a minimal graph on a fixed domain
  $\Omega\subset\M$ which extends continuously to $\overline\Omega$,
  and let $T$ be a geodesic arc in~$\partial\Omega$.
  \begin{itemize}
  \item[(i)] If $\{u_n\}$ diverges uniformly to $+\infty$ on compact
    sets of $T$ while remaining uniformly bounded on compact sets of
    $\Omega$, then $F_{u_n}(T)\to|T|$.
  \item[(ii)] If $\{u_n\}$ diverges uniformly to $+\infty$ on
    compact sets of $\Omega$ while remaining uniformly bounded on compact
    sets of $T$, then $F_{u_n}(T)\to-|T|$.
  \end{itemize}
\end{lemma}

The following result is adapted to the situation of the next section. The
boundary of a domain $\Omega$ is finitely piecewise smooth and locally
convex if it is composed of a finite number of open smooth arcs which are
convex towards $\Omega$, together with their endpoints. These endpoints
are called
the vertices of $\Omega$. 

\begin{theorem}[Divergence set theorem]
  \label{th:divergenceset}
Let $\Omega\subset\M$ be a bounded domain with finitely piecewise smooth
and locally convex boundary. Let $\{u_n\}$ be an increasing (resp.
decreasing) sequence of minimal graphs on $\Omega$. For every open smooth
arc $C\subset\partial\Omega$, we assume that, for every $n$, $u_n$ extend
continuously on $C$ and either $u_n|_C$ converges to a continuous function
or $u_n|_C\nearrow+\infty$ (resp. $u_n|C\searrow-\infty$). Let ${\cal V}$
be the divergence set associated to  $\{u_n\}$ 
  \begin{enumerate}
  \item The boundary of ${\cal V}$ consists of a finite set of
    non-intersecting interior geodesic chords in $\Omega$ joining two
    vertices of $\partial\Omega$, together with geodesics in
    $\Omega$.
    \item A component of ${\cal V}$ cannot only consist of an isolated
      point nor an interior chord.
    \item No two interior chords in $\partial{\cal V}$ can have a
      common endpoint at a convex corner of ${\cal V}$.
  \end{enumerate}
\end{theorem}

\begin{theorem}[Maximum principle for bounded domains]\label{th:max}
  %p�ginas 4,5,6 del art�culo de Ana Lucia
  Let $\Omega\subset\M$ be a bounded domain, and $E\subset\partial\Omega$ a
  finite set of points. Suppose that $\partial\Omega\backslash E$
consists of smooth
  arcs $C_k$, and let $u_1,u_2$ be minimal graphs on $\Omega$
  which extend continuously to each $C_k$. If $u_1\leq u_2$ on
  $\partial\Omega \backslash E$, then $u_1\leq u_2$ on~$\Omega$.
\end{theorem}

\begin{theorem}[Boundary values lemma]
  Let $\Omega\subset\M$ be a domain and let $C$ be a compact convex
  arc in $\partial\Omega$. Suppose $\{u_n\}$ is a sequence of
  minimal graphs on $\Omega$ converging uniformly on compact subsets
  of $\Omega$ to a minimal graph $u:\Omega\to\R$. Assume each $u_n$ is
  continuous in $\Omega\cup C$ and $\{u_n|_C\}$ converges uniformly
  to a function $f$ on $C$. Then $u$ is continuous in $\Omega\cup C$
  and $u|_C=f$.
\end{theorem}

%%%%%%%%%%%%%%%%%%%
%%%%%%%%%%%%%%%%%%%

\section{A general Jenkins-Serrin theorem on $\M\times\R$}
\label{secth1}

%%%%%%%%%%%%%%%%%%%
%%%%%%%%%%%%%%%%%%%

Let $\Omega\subset\M$ be a bounded domain whose boundary consists of a
finite number of open geodesic arcs
$A_1,\cdots,A_{k_1},B_1,\cdots,B_{k_1}$ and a finite number of open
convex arcs $C_1,\cdots,C_{k_3}$ (convex towards $\Omega$), together
with their endpoints.  We mark the $A_i$ edges by $+\infty$, the $B_i$
edges by $-\infty$, and assign arbitrary continuous data $f_i$ on the
arcs $C_i$.

\begin{definition}
  {\rm We define a {\it solution for the Dirichlet problem on
      $\Omega$} as a minimal graph $u:\Omega\to\R$ which assumes the
    above prescribed boundary values on $\partial\Omega$.}
\end{definition}

Our aim in this section is to solve this Dirichlet problem on
$\Omega$. We assume that no two $A_i$ edges and no two $B_i$ edges
meet at a convex corner (see Remark~\ref{rem:corner}). When $\Omega$
is geodesically convex, this was done in~\cite{pin1}; in general we
need another condition on the $\partial\Omega$. We assume the
following technical condition is satisfied:
\begin{quote}
\begin{itemize}
\item[(C1)] If $\{C_i\}_i=\emptyset$, then neither
  $\cup_{i=1}^{k_1}\overline{A_i}$ nor
  $\cup_{i=1}^{k_2}\overline{B_i}$ is a connected subset of
  $\partial\Omega$.
\end{itemize}
\end{quote}

We will say that a domain $\Omega$ as above is a \textit{Scherk domain}.
We notice that the hypothesis (C1) implies that $k_1\ge 2$ and
$k_2\ge2$ when $\{C_i\}_i=\emptyset$. We remark that (C1) is always
satisfied when $\M=\R^2,\H^2$.

Condition (C1) is not necessary for the existence of a solution to the
Dirichlet problem on $\Omega$ (see Remark~\ref{rem:esf}) but we need to
assume this for our proof.

\begin{claim}\label{cl:cond1}
In particular, condition \mbox{\rm (C1)} holds when there exists a
component $\Gamma$ of $\partial\Omega$ and a strongly
geodesically convex\footnote{A set $D\subset\M$ is said to be
\emph{strongly geodesically convex} when, for every $p,q\in \overline{D}$,
there exists a unique length-minimizing geodesic arc $\gamma$ in $\M$
joining $p,q$ and $\gamma\subset \overline{D}$; moreover, $\gamma$ is the
only geodesic arc in $\overline{D}$ joining $p,q$.} domain
$\Omega'\subset\M$ containing $\Omega$ such that $\partial\Omega'=\Gamma$.
\end{claim}
\begin{proof}
Suppose $\{C_i\}_i=\emptyset$. Since $\Gamma$ is the boundary of
$\Omega'$ and~$\overline{\Omega'}$ is geodesically convex, we can
rename the $A_i,B_i$ edges so that $\Gamma=A_1$ or $\Gamma=B_1$ or
$\Gamma=A_1\cup B_1\cup\cdots\cup A_k\cup B_k$ (cyclically ordered). The
first two cases are not allowed: in fact, in that cases $A_1$ or $B_1$
would be closed and two points on it would be joined by two geodesic arcs
in $\Gamma\subset\overline{\Omega'}$.

In the third case, we have $k\ge 2$. If $k=1$, the common endpoints of
$A_1$ and $B_1$ are joined by two geodesic arcs, $A_1$ and $B_1$, in
$\overline{\Omega'}$ which is impossible. Thus $k\ge 2$ and (C1) holds.
\end{proof}

A polygonal domain ${\cal P}$ is said to be {\it inscribed
  in $\Omega$} when ${\cal P}\subset\Omega$ and its vertices are drawn
from the set of endpoints of the
$A_i,B_i,C_i$ edges.  Given a polygonal domain ${\cal P}$ inscribed in $\Omega$,
we denote by $\gamma$ the perimeter of $\partial{\cal P}$, and by $\alpha$
(resp.~$\beta$) the total length of the edges $A_i$ (resp. $B_i$)
lying in $\partial{\cal P}$.

\begin{theorem}
\label{th1}
Let $\Omega$ be a Scherk domain. If the family $\{C_i\}_i$ is non-empty,
there exists a solution to the Dirichlet problem on $\Omega$ if and only if
\begin{equation}
  \label{hipJS}
  2\a<\g\qquad\mbox{and }\qquad 2\be<\g
\end{equation}
for every polygonal domain ${\cal P}$ inscribed in $\Omega$. Moreover,
such a solution is unique, if it exists.

When $\{C_i\}_i$ is empty, there is a solution to the Dirichlet
problem for $\Omega$ if and only if $\alpha=\beta$ when ${\cal P}=\Omega$, and
inequalities in (\ref{hipJS}) hold for all other polygonal domains
inscribed in $\Omega$. Such a solution is unique up to an additive
constant, if it exists.
\end{theorem}

\begin{remark}\label{rem1}\mbox{}
  \begin{enumerate}
    \item The Scherk domain $\Omega$ need not be convex, even when
      there are no $A_i$ and $B_j$ edges. There are no conditions in
      the latter case; the solution need not be continuous at the
      vertices. 
  \item Theorem~\ref{th1} corresponds to Theorem 4 in~\cite{jes1}, in
    the case $\M=\R^2$.
  \item Theorem~\ref{th1} has been proven, when $\Omega$ is a
    geodesically convex domain, by Nelli and Rosenberg~\cite{NeRo2}
    (in the case $\M=\H^2$) and by Pinheiro~\cite{pin1}.
  \end{enumerate}
\end{remark}

\begin{proof}
  The uniqueness part in Theorem~\ref{th1} can be proven exactly as
  in~\cite{pin1}.
    Let us now prove the conditions of Theorem~\ref{th1} are necessary
for existence.
  Suppose there is a minimal graph $u$ solving the Dirichlet
problem. 
  When $\{C_i\}_i=\emptyset$ and ${\cal P}=\Omega$, using
  Lemma~\ref{lem:flux} we have
  \[
  \textstyle{
    \alpha=\sum_i|A_i|=\sum_i F_u(A_i)=-\sum_i F_u(B_i)=\sum_i|B_i|=\beta ,}
  \]
  as we wanted to prove. In the other case, again by Lemma~\ref{lem:flux},
  we obtain:
  \begin{itemize}
  \item $\sum_{A_i\subset\partial{\cal P}} F_u(A_i)+
    \sum_{B_i\subset\partial{\cal P}} F_u(B_i)+
    F_u(\partial{\cal P}-\cup_i A_i-\cup_i B_i)=0$.
  \item $\sum_{A_i\subset\partial{\cal P}} F_u(A_i)=
    \sum_{A_i\subset\partial{\cal P}} |A_i|=\alpha$.
  \item $ \sum_{B_i\subset\partial{\cal P}} F_u(B_i)=
    -\sum_{B_i\subset\partial{\cal P}} |B_i|=-\beta$.
  \item $ |F_u(\partial{\cal P}-\cup_i A_i- \cup_i B_i)|<
    \gamma-\alpha-\beta$.
    \end{itemize}
    From all this, $|\alpha-\beta|<\gamma-\alpha-\beta$, so
    $2\alpha<\gamma$ and $2\beta<\gamma$, as desired.\\

  Finally, let us prove the conditions are sufficient. We distinguish the
  following cases:\\

    \noindent {\bf $\star$ First case: Suppose that the families
    $\{A_i\}_i,\{B_i\}_i$ are both empty. }\\
  In this case, Theorem~\ref{th1} is proven, exactly as in~\cite{jes1}
  for $\M=\R^2$, by means of the Perron process (see~\cite{gt1,jes1}),
  using the fact that the solution to the Dirichlet problem exists for
  small geodesic disks~\cite{pin1} and a standard barrier argument (a
  barrier exists at every convex boundary point, see~\cite{pin1}).\\

  \noindent
  {\bf $\star$ Second case: Suppose $\{B_i\}_i=\emptyset$
    and each $f_i$ is bounded below. }\\
  Using the previous step, there exists, for every $n\in\N$, a unique
  minimal graph ${u_n:\Omega\to\R}$ such that:
  \[
  \left\{\begin{array}{lllll}
      u_n=n      & \mbox{, on the } A_i \mbox{ edges.}\\
      u_n=\min\{n,f_i\} & \mbox{, on the } C_i \mbox{ edges.}\\
    \end{array}\right.
  \]
  From the maximum principle for bounded domains (Theorem~\ref{th:max}), we
deduce
  that $\{u_n\}$ is a non-decreasing sequence.  Thus
  Lemma~\ref{lem:convexhull} and Theorem~\ref{th:divergenceset} assure
  that, if it is non-empty, the divergence set ${\cal V}$ of
  $\{u_n\}$ consists of a finite number of polygonal domains
  inscribed in $\Omega$. Assume that ${\cal V}$ is connected
  (otherwise, we will similarly argue on each component of ${\cal
    V}$).  By Lemma~\ref{lem:flux}, the flux of $u_n$ along $\partial{\cal
V}$ vanishes; this is,
$$
  \sum_{A_i\subset\partial{\cal V}} F_{u_n}(A_i)+ F_{u_n}(\partial{\cal V}-
  \cup_i A_i)=0 .
$$
  On the other hand,
  Lemma~\ref{lem:limitflux} says that $F_{u_n}(\partial{\cal V}-\cup_i
  A_i)\rightarrow -(\gamma-\alpha)$ as $n\rightarrow +\infty$.  Since
$\sum_{A_i\subset\partial{\cal V}}
  |F_{u_n}(A_i)|\leq\alpha$, we obtain $2\alpha-\gamma\geq 0$, which
  contradicts~(\ref{hipJS}). Hence ${\cal V}=\emptyset$, and
  $\{u_n\}$ converges uniformly on compact sets of $\Omega$ to a
  minimal graph $u:\Omega\to\R$. The desired boundary
  conditions for $u$ are obtained from standard barrier arguments.\\

  Theorem~\ref{th1} can be proven analogously when $\{A_i\}_i$ is
  empty and each $f_i$ is bounded above.

  \noindent
  {\bf $\star$ Third case: Suppose
    $\{C_i\}_i\neq\emptyset$. }\\
  By the previous step, there exist (unique) minimal graphs $u^+,
  u^-, u_n:\Omega\to\R$ with the following boundary values:
  \[
  \left\{\begin{array}{lllll}
      u^+=+\infty &  ,\ u^-=0       & \mbox{and} & u_n=n      & \mbox{, on the } A_i \mbox{ edges,}\\
      u^+=0       &  ,\ u^-=-\infty & \mbox{and} & u_n=-n     & \mbox{, on the } B_i\mbox{ edges,}\\
      u^+=f_i^+   &  ,\ u^-=f_i^-   & \mbox{and} & u_n=f_{i,n} & \mbox{, on the } C_i \mbox{ edges,}\\
    \end{array}\right.
  \]
  where $f_i^+=\max\{0,f_i\}$, $f_i^-=\min\{0,f_i\}$ and $f_{i,n}$
  denotes the function $f_i$ truncated above and below by $n$ and
  $-n$, respectively.  By Theorem~\ref{th:max}, $u^-\leq u_n\leq u^+$,
for every
  $n$. Using the compactness theorem (Theorem~\ref{th:compactness})
  and a diagonal process we can extract a subsequence of $\{u_n\}$
  which converges on compact sets of $\Omega$ to a minimal graph $u$.
  The desired boundary conditions for $u$ are obtained from
  standard barrier arguments.  \\

  \noindent
  {\bf $\star$ Fourth case: Suppose $\{C_i\}_i=\emptyset$. }\\
  From the first case, we know there exists for each $n\in\N$ a
  minimal graph $v_n:\Omega\to\R$ such that
  \[
  \left\{\begin{array}{ll}
      v_n=n     & \mbox{, on the } A_i \mbox{ edges.}\\
      v_n=0     & \mbox{, on the } B_i \mbox{ edges.}\\
    \end{array}\right.
  \]
  And the maximum principle implies that $0\le
v_n\le n$. 
  For every $c\in(0,n)$, we define
  \[
  E_c=\{p\in D\ |\ v_n(p)>c\} ,\quad F_c=\{p\in D\ |\ v_n(p)<c\} ,
  \]
  and denote by $E_c^i$ (resp. $F_c^i$) the component of $E_c$ (resp.
  $F_c$) whose closure contains the edge $A_i$ (resp. $B_i$). From the
  maximum principle for bounded domains, we can deduce $E_c=\cup_i E_c^i$
and
  $F_c=\cup_i F_c^i$.
  
  Condition (C1) ensures that the set $F_c$ (resp. $E_c$) is
  disconnected for $c=\varepsilon$ (resp.  $c= n-\varepsilon$), with
  $\varepsilon>0$ small enough. On the other hand, $F_c$ is connected
  when $c=n-\varepsilon$ for $\varepsilon>0$ small enough, so we can
  define
  \[
  \mu_n=\inf\{c\in(0,n)\ |\ \mbox{ the set } F_c \mbox{ is
    connected}\},
  \]
  and $u_n=v_n-\mu_n$.

  In order to prove that a subsequence of $\{u_n\}$ converges, let
  us consider the auxiliary functions
  \[
  u^+=\max_i\{u_i^+\}\, , \qquad u^-=\min_i\{u_i^-\}\, ,
  \]
  where $u_i^+,u_i^-:\Omega\to\R$ are the unique
  minimal graphs given by
\[
  \begin{array}{lcr}
    \left\{\begin{array}{ll}
        u_i^+=+\infty      & \mbox{, on } \cup_{i'\neq i} A_{i'}\\
        u_i^+=0     & \mbox{, on }  (\cup_j B_j)\cup A_i\\
      \end{array}\right.
    & \mbox{}\hspace{1cm} &
    \left\{\begin{array}{ll}
      u_i^-=-\infty      & \mbox{, on } \cup_{i'\neq i} B_{i'}\\
      u_i^-=0     & \mbox{, on } (\cup_j A_j)\cup B_i\\
    \end{array}\right.
  \end{array}
  \]
  (such functions $u_i^+,u_i^-$ exist thanks to the second case
  studied previously).

  Observe that, by definition of $\mu_n$, both $E_{\mu_n},F_{\mu_n}$ are
  disconnected.  In particular, for every $i_1$, there exists a $i_2$
  such that $E_{\mu_n}^{i_1}\cap E_{\mu_n}^{i_2}=\emptyset$, and we
  obtain, applying the maximum principle,
  \[
  0\leq u_n|_{E_{\mu_n}^{i_1}}\leq u_{i_2}^+|_{E_{\mu_n}^{i_1}} .
  \]
  Similarly, for every $j_1$, there exists a $j_2$ such that
  $F_{\mu_n}^{j_1}\cap F_{\mu_n}^{j_2}=\emptyset$, and
  \[
  u_{j_2}^-|_{F_{\mu_n}^{j_1}} \leq u_n|_{F_{\mu_n}^{j_1}}\leq 0.
  \]
  From this it is not very difficult to prove that $u^-\leq u_n\leq
  u^+$. Hence, the compactness theorem ensures that a subsequence of
  $\{u_n\}$ converges uniformly on compact subsets of $\Omega$ to a
  minimal graph~$u$. Let us check that $u$ satisfies the desired
  boundary conditions.

  Suppose that, after passing to a subsequence, $\{\mu_n\}$
  converges to some $\mu_\infty<+\infty$. Hence, $u=-\mu_\infty$ on
  each $B_i$ and $u$ diverges to $+\infty$ when we approach $A_i$
  within $\Omega$. From Lemma~\ref{lem:flux}, we get
  \[
  \textstyle{
    \sum_i F_u(A_i)+\sum_i F_u(B_i)=F_u(\partial\Omega)=0 ,}
  \]
  \[
  \textstyle{
    \sum_i F_u(A_i)=\alpha \qquad\mbox{and} \qquad
    |\sum_i F_u(B_i)|<\beta,}
  \]
  which contradicts the assumption $\alpha=\beta$. Thus the whole
  sequence $\{\mu_n\}$ diverges to $+\infty$.  Analogously, we can
  prove that $n-\mu_n\to+\infty$ as $n\to+\infty$, and
  Theorem~\ref{th1} is proven.
\end{proof}

\begin{remark}\label{rem:esf}
  The following example shows condition (C1) is not necessary:
  Consider a hemisphere $\Omega_0\subset\esf^2$ and a geodesic
  triangle $T_ 1\subset\Omega_0$. By Theorem~\ref{th1}, there exists a
  minimal graph on $\Omega_0-T_1$ with boundary data $0$ on
  $\partial\Omega_0$ and $+\infty$ on $\partial T_1$ (up to its
  vertices). Considering the $\pi$- rotation about $\partial\Omega_0$,
  we get a minimal graph defined on the sphere with two geodesic
  triangles $T_1,T_2$ removed which has boundary data $+\infty$ on the
  edges of $\partial T_1$ and $-\infty$ on the edges of $\partial
  T_2$, see Figure~\ref{fig:esf}.
\end{remark}

\begin{figure}\begin{center}
\epsfysize=6cm \epsffile{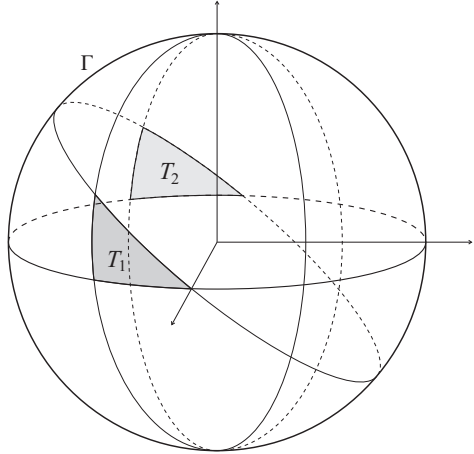}
\end{center}
\caption{$\Omega=\esf^2-(T_1\cup T_2)$ does not satisfies the
  condition (C1) when $\partial T_1=A_1\cup A_2\cup A_3$ and $\partial
  T_2=B_1\cup B_2\cup B_3$.}
\label{fig:esf}
\end{figure}

Before ending this section, let us give a result which is the converse of
statement $(iii)$ in Lemma \ref{lem:flux}.

\begin{lemma}\label{lem:xulf}
Let $u$ be a minimal graph on a domain $\Omega\subset\M^2$. Let
$T\subset\partial\Omega$ be a geodesic arc such that $F_u(T)=|T|$ (resp.
$F_u(T)=-|T|$). Then $u$ takes on $T$ the boundary value $+\infty$ (resp 
$-\infty$).
\end{lemma}

\begin{proof}
Let us consider $p\in T$, and $\Omega'$ be the set of points in $\Omega$ at
distance less than $\delta$ from $p$ ($\delta$ is chosen very small),
$\Omega'$ is a half-disk. Let $T'$ be $T\cap\partial\Omega'$, we have
$F_u(T')=|T'|$ and the other part of $\partial \Omega'$ is strictly
convex. From Theorem \ref{th1}, there exists on $\Omega'$ a minimal graph
$v$ with $u=v$ on $\partial\Omega'\backslash T'$ and $v=+\infty$ on $T'$.
The lemma is proved if we show that $u=v$. 

If the lemma is not true, we can assume that $\{u<v-\ve\}$ is
nonempty; where $\ve$ is chosen to be a regular value of $v-u$. Let $O$
denote $\{u<v-\ve\}$. Let $C$ be the connected component of the complement
of $O$ which has $\partial\Omega'\backslash T'$ in its boundary and we
consider $O'$ the complement of $C$: we have $O\subset O'$ and $\partial
O'\subset \partial O \cup T'$. Let $q$  be a point in $\partial O'
\cap\Omega'$. For $\mu>0$, let $O'(\mu)$ be the set of point $O'$ at
distance larger than $\mu$ from $T'$. Let $q_1$ and $q_2$ be the
endpoints of the connected component of $\partial O'(\mu)\cap
\partial O'$ which contains $q$. Let $p_i$ be the projection of $q_i$ on
$T'$. Let $\widetilde{O}(\mu)$ be the domain bounded by the segments
$[q_1,p_1]$, $[p_1,p_2]\subset T'$, $[p_2,q_2]$ and the boundary component
of $O'(\mu)$ between $q_2$ and $q_1$. On this last component
$\Gamma(\mu)$ the vector $X_u-X_v$ points outside $\widetilde{O}(\mu)$.
Since $F_u(\partial \widetilde{O}(\mu)) =0= F_v(\partial
\widetilde{O}(\mu))$, we have:
\begin{align*}
0<\int_{\Gamma(\mu)}\langle X_u-X_v,\nu\rangle&= -
\int_{[p_1,q_1]\cup[p_2,q_2]} \langle X_u-X_v,\nu\rangle - 
\int_{[p_1,p_2]}\langle X_u-X_v,\nu\rangle\\
&\le 4\mu-\int_{[p_1,p_2]}\langle X_u-X_v,\nu\rangle 
\end{align*}
By hypothesis on $u$ and $v$ and Lemma \ref{lem:flux}$-(iii)$, the last
term vanishes; moreover the integral on $\Gamma(\mu)$ increases as $\mu$
goes to $0$ (see Lemma 2 in \cite{ck1}). Thus we have a contradiction and
$u=v$.
\end{proof}

%%%%%%%%%%%%%%%%%%%
%%%%%%%%%%%%%%%%%%%

\section{A particular case:  $\M=\H^2$}\label{secH2}

%%%%%%%%%%%%%%%%%%%
%%%%%%%%%%%%%%%%%%%

In the rest of the paper we study the Dirichlet problem for unbounded
domains in $\H^2$. 

Collin and Rosenberg~\cite{cor2} have extended
Theorems~\ref{th:divergenceset} and~\ref{th:max} to some unbounded
domains. More precisely, they consider simply connected domains
$\Omega\subset\H^2$ whose boundary consists of finitely many ideal
geodesics and finitely many complete convex arcs (convex towards
$\Omega$) together with their endpoints at infinity, $\Omega$
satisfying the following assumption:

\begin{figure}\begin{center}
\epsfysize=5cm \epsffile{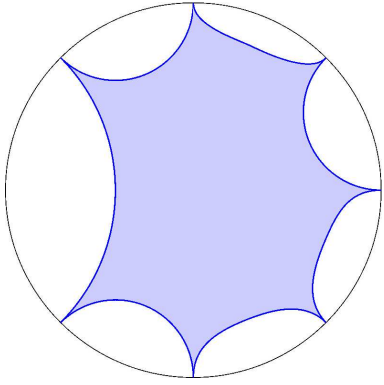}
\end{center}
\caption{A domain $\Omega\subset\H^2$ satisfying condition (C-R).}
\label{fig:CR}
\end{figure}

\begin{quote}
  (C-R)\quad If $C\subset \partial\Omega$ is a convex arc with endpoint
  $p\in\partial_\infty\H^2$, then the other arc $\g$ of
  $\partial\Omega$ having $p$ as an endpoint is asymptotic to $C$ at
  $p$; i.e., if $\{x_n\}$ is a sequence in $\g$ converging to $p$,
  then ${\rm dist}_{\H^2}(x_n,C)\to 0$ (see Figure~\ref{fig:CR}).
\end{quote}
They solve the Dirichlet problem for such domains.  The same results
without assuming $\Omega$ is simply connected can be obtained from
Theorem~\ref{th1}, following Collin and Rosenberg's ideas. Our aim is to
weaken the hypotheses on $\Omega$, in particular the (C-R)
hypothesis. Also we will allow $\Omega$ to have arcs in
$\partial_\infty\H^2$ in its closure.

%%%%%%%%%%%%%%%%%%%

\subsection{Minimal graphs over unbounded domains}

%%%%%%%%%%%%%%%%%%%

%%%%%%%%%%%%%%
\subsubsection{First examples}
%%%%%%%%%%%%%%

Let $p$ be a point in $\partial_\infty 
\H^2$. We consider the half-plane model for the hyperbolic plane,
$\H^2=\{(x,y)\in\R^2\ |\ y>0\}$ with metric $\langle\ ,\ \rangle
=\frac{1}{y^2}\, g_0$, where $g_0$ is the Euclidean metric and assume that
$p$ is the point of coordinates $(0,0)$. For $(\phi,\theta)\in
\R\times(0,\pi)$ we consider the point $q=(e^\phi\cos\theta,
e^\phi\sin\theta)\in \R\times \R_+^*=\H^2$. We will call
$(\phi,\theta)$ the polar coordinates of $q$ centered at $p$. In these new
coordinates, the hyperbolic metric becomes $\frac{1}{\sin^2\theta}
(d\phi^2+ d\theta^2)$; the coordinates $(\phi,\theta)$ are conformal. 

We notice that there are several polar coordinates centered at $p$
\textit{i.e.} given a point $q \in \H^2$ there exists one 
hyperbolic isometry fixing $p$ such that the polar coordinates centered at
$p$ of $q$ becomes $(0,\pi/2)$. The curves $\{\phi=constant\}$ are
geodesics. The curve $\{\theta=\pi/2\}$ is also a geodesic  of $\H^2$ and,
 for any $\theta_0\in(0,\pi)$, the curve $\{\theta=\theta_0\}$ is 
equidistant
to this geodesic; we denote by
\begin{equation}\label{def:distance}
d_{\theta_0}=
\left|\int_{\theta_0}^{\pi/2}\frac{d\theta}{\sin\theta}\right|
\end{equation}
the distance between the geodesic $\{\theta=\pi/2\}$ and its equidistant
$\{\theta=\theta_0\}$.

A minimal graph $u$ which takes constant values on the equidistant curves
to the geodesic $\{\theta=\pi/2\}$ can be written
$u(\phi,\theta)=f(\theta)$, where $f$ satisfies the
following differential equation (see Appendix \ref{appendix}):
$$
\dfrac{d}{d\theta}\left(\frac{f'}{\sqrt{
1+\sin^2\theta\left|f'\right|^2}} \right)=0
$$
Thus, by integrating this equation with $f(0)=0$, we get minimal surfaces
that were first obtained by Sa Earp~\cite{e} and Abresch (see Appendix
\ref{appendix}).
\begin{lemma}\label{lem:barrier}
Let $\theta_0\in(0,\pi/2]$. There is a minimal graph $h_{\theta_0}$
defined on the domain $\Omega_{\theta_0}=\{0<\theta<\theta_0\}$ which takes
constant values on the equidistant curves to
$\{\theta=\pi/2\}$, have boundary data $0$ on the boundary arc
$\{\theta=0\}$ and satisfies $\dfrac{d h_{\theta_0}}{d\nu}=+\infty$ on
$\{\theta=\theta_0\}$ ($\nu$ is the outer unit normal to
$\partial\Omega_{\theta_0}$). When
$\theta_0<\pi/2$, $h_{\theta_0}$ takes a constant finite value on
$\{\theta=\theta_0\}$ and $h_{\pi/2}$ diverges to $+\infty$ on the
geodesic $\{\theta=\pi/2\}$
\end{lemma}

In the half-plane model, the minimal graph $h_{\pi/2}$ is defined on
$\R_+^*\times\R_+^*$ by 
\begin{equation}\label{expressionh}
h_{\pi/2}(x,y)=\ln\frac{\sqrt{x^2+y^2}+y}{x}
\end{equation}
Then if $\Omega$ is a domain bounded by a geodesic and
an arc in $\partial_\infty\H^2$, Lemma \ref{lem:barrier} gives a minimal
graph $h$ over $\Omega$ with value $0$ on the arc in $\partial_\infty\H^2$
and $h=+\infty$ on the geodesic. We notice that $\pm h+M$  is a minimal
graph over $\Omega$ with value $M$ on the arc in $\partial_\infty\H^2$ and
$\pm\infty$ on the geodesic. These minimal graphs are examples of solutions
to a Dirichlet problem that can be recovered by the work of Collin and
Rosenberg in \cite{cor2}.

In the following, we want to generalize such examples. The above surfaces
will be used as barriers to study boundary values and uniqueness. As above,
the domains $\Omega$ we shall study have arcs in $\partial_\infty\H^2$ as
boundary; thus we shall denote by $\partial\Omega$ the boundary of
$\Omega$ in $\H^2$ and by $\partial_\infty\Omega$ the boundary of $\Omega$
in the compactified space $\H^2\cup\partial_\infty\H^2$;
$\overline{\Omega}^\infty$ will denote the closure of $\Omega$ in $\H^2\cup
\partial_\infty\H^2$.

%%%%%%%%%%%%%%
\subsubsection{Convergence of sequences of minimal graphs}
%%%%%%%%%%%%%%

In this section, we solve the Dirichlet problem in a more general
setting, where a maximum principle is not necessarily
satisfied (see Section~\ref{secPrincMax}). We cannot then apply the
method developed by Jenkins and Serrin to solve the Dirichlet problem
on $\Omega$, since we cannot assure the monotonicity of the
constructed graphs $u_n$ in the third step of the proof (see the third
case ``$\{C_i\}\neq\emptyset$'' in the proof of Theorem~\ref{th1}).
We now study the convergence of a (non necessarily monotone) sequence
of minimal graphs on~$\Omega$.

Let $\Omega\subset\H^2$ be a domain whose boundary $\partial_\infty
\Omega$ is piecewise smooth
(possibly with some arcs at $\partial_\infty\H^2$).  Given a sequence
$\{u_n\}$ of minimal graphs on $\Omega$, we define the {\it
  convergence domain} of the sequence $\{u_n\}$ as
\[
{\cal B}=\left\{p\in\Omega\ |\ \{|\nabla u_n(p)|\} \text{ is
bounded}\right\},
\]
and the \textit{divergence set} of $\{u_n\}$ as
\[
{\cal D}=\Omega-{\cal B} .
\]
We remark that, in Theorem~\ref{th:monotone}, we have already defined a
notion of convergence and divergence set for monotone sequences. In the
following, we only use these new definitions.

The following lemma gives us a local description of the convergence
domain ${\cal B}$ and the divergence set ${\cal D}$ that justifies their
names.  $G(u_n)$ will denote the graph of $u_n$, and $N_n(p)$ the
downward pointing normal vector to $G(u_n)$ at the point $(p,u_n(p))$;
i.e. $N_n=(X_{u_n},\frac{-1}{W_{u_n}})$. For writting this, we use a
vertical translation to identify the tangent space $T(\H^2\times \R)$ with
$T\H^2\times \R$. In fact, in the following, we often use vertical
translations to identify the tangent spaces.

\begin{lemma}\label{lem:conv_div}\mbox{}
  \begin{enumerate}
  \item Given $p\in{\cal B}$, there exists a subsequence of
    $\{u_n-u_n(p)\}$ converging uniformly to a minimal graph in a
    neighborhood of $p$ in $\Omega$. The size of the neighborhood depends
only on the distance from $p$ to $\partial\Omega$ and an upper-bound for
$\{|\nabla u_n(p)|\}$. Also, ${\cal B}$ open follows
    from curvature estimates.
  \item If $p\in{\cal D}$, there exists a compact geodesic arc
    $L_p(\delta)\subset\Omega$ of length $2\delta$ centered at $p$,
$\delta>0$ only depends on ${\rm dist}_{\H^2}(p,\partial\Omega)$, such
that, after passing to a subsequence, $\{N_n(q)\}$ converges to a
horizontal vector orthogonal to $L_p(\delta)$ at every point $q\in
    L_p(\delta)$. 
\end{enumerate}
\end{lemma}
\begin{proof}
  Fix $p\in\Omega$, and define $v_n=u_n-u_n(p)$.  We denote by
  $G(v_n)$ the graph of $v_n$. Observe that, for any $q\in\Omega$, the
  downward pointing normal vector to $G(v_n)$ at $Q=(q,v_n(q))$
  coincides with $N_n(q)$, and that both the convergence and
  divergence sets associated to $\{v_n\}$ and $\{u_n\}$ coincide.  The
  distance from $P=(p,0)$ to the boundary of $G(v_n)$ is bigger than
  or equal to $d={\rm dist}_{\H^2}(p,\partial\Omega)$.  Hence we
  deduce from Schoen's curvature estimates~\cite{sc3} that there
  exists $\delta>0$ depending on $d$ such that a neighborhood of $P=(p,0)$
in $G(v_n)$
  is a graph of uniformly bounded height and slope 
  over the disk $\D_{n}(\delta)\subset T_P G(v_n)$ of radius $\delta$
  centered at the origin of $T_P G(v_n)$ (see
  \cite{pro2}, Lemma 4.1.1, for more details). By graph here we mean a
  graph in geodesic coordinates, orthogonal to $\D_n(\delta)$. We call
  $G_n(p,\delta)$ such a graph.
  
  Suppose $p\in{\cal B}$. Since $\{|\nabla u_n(p)|\}$ is uniformly
  bounded, a subsequence of $\{N_n(p)\}$ converges to a non-horizontal
  vector, so the tangent planes $T_P G(v_n)$ converge to a
  non-vertical plane $\Pi$, and the disks $\D_{n}(\delta)$ converge to
  a disk $\D(\delta)\subset\Pi$ of radius $\delta$.  From standard
  arguments (see \cite{pro2}, Theorem 4.1.1) we deduce that a
  subsequence of $\{G_n(p,\delta)\}$ converges to a minimal graph
  $G(p,\delta)$ over $\D(\delta)$.  Hence there exists a disk
  $D(p,\widetilde\delta)\subset\Omega$ of radius
  $\widetilde\delta\in(0,\delta]$ such that
  $\{v_n|_{D(p,\widetilde\delta)}\}$ is uniformly bounded. After
  passing to a subsequence, $\{v_n|_{D(p,\delta)}\}$ converges
  uniformly on compact subsets of $D(p,\widetilde{\delta})$ to a minimal
  (vertical) graph.  This proves {\it 1}.

  Now assume $p\in{\cal D}$.  Since $\{|\nabla u_n(p)|\}$ is
  unbounded, we can take a subsequence of $\{u_n\}$ so that $|\nabla
  u_n(p)|\to+\infty$ and $\{N_n(p)\}$ converges to a horizontal
  vector. In particular, the tangent planes $T_P(G(v_n))$ converge to
  a vertical plane $\Pi$, and a subsequence of $\{G_n(p,\delta)\}$
  converges to a minimal graph $G(p,\delta)$ over a disk
  $\D(\delta)\subset\Pi$ of radius $\delta$ centered at $P$. The graph
  $G(p,\delta)$ is tangent to $\Pi$ at $P$.  The following argument
  follows the ideas in~\cite{hrs2}, Claim 1: If
  $G(p,\delta)\not\subset\Pi$, then $G(p,\delta)\cap\Pi$ consists of
  $k\geq 2$ smooth curves meeting transversally at $P$. In particular,
  there are parts of $G(p,\delta)$ on both sides of $\Pi$. Thus there
  are points in $G(p,\delta)$ where the normal vector points up and
  points where the normal points down. But this is impossible, since
  $G(p,\delta)$ is the limit of vertical graphs. Therefore,
  $G(p,\delta)\subset\Pi$.

  We call $L_p(\delta)$ the geodesic $G(p,\delta)\cap(\H^2\times\{0\})$,
whose
  length is $2\delta$.  We can deduce that the tangent planes of $G(v_n)$
at
  $(q,v_n(q))$ converge to $\Pi$, for every $q\in L_p(\delta)$ (for precise
  details, see~\cite{mazet0,mazet5}), which completes the
  proof of Lemma~\ref{lem:conv_div}.
\end{proof}

The next lemma shows ${\cal D}=\cup_{i\in I} L_i$, where each $L_i$ is
a component of the intersection of a ideal geodesic in $\H^2$
with~$\Omega$. The geodesics $L_i$ are called {\it divergence lines}.

\begin{lemma}\label{lem:continuation}
  Given $p\in{\cal D}$, there exists a geodesic $L\in\Omega$ joining
  points in $\partial_\infty\Omega$ (possibly at $\partial_\infty\H^2$)
which
  passes through $p$ and such that, after passing to a subsequence,
  $\{N_n|_L\}$ converges to a horizontal vector orthogonal to $L$ (in
  particular, $L\subset{\cal D}$). In fact, $L$ is the geodesic
  containing $L_p(\delta)$.
\end{lemma}
\begin{proof}
  Let $L_p=L_p(\delta)$ be the geodesic arc given in
  Lemma~\ref{lem:conv_div}-{\it 2}, and $L$ be the geodesic in
  $\Omega$ joining points in $\partial\Omega$ which contains $L_p$.
  For every $q$, we denote by $[p,q]\subset L$ the closed geodesic arc
  in $L$ joining $p,q$.  Define
  \[
  \Lambda=\left\{q\in L\ \Big|\ \begin{array}{ll}
    \mbox{ there exists a subsequence of } \{u_n\}\
    \mbox{ such that }\\
    N_n|_{[p,q]}\ \mbox{ becomes horizontal and
      orthogonal to } L
    \end{array}\right\}.
  \]
  Clearly, $p\in\Lambda$ so $\Lambda\neq\emptyset$.  Let us prove
  $\Lambda$ is open in $L$.  Take $q\in\Lambda$, and denote by
  $\{u_{\sigma(n)}\}$ its associated subsequence given in the
  definition of $\Lambda$. Since $\Lambda\subset{\cal D}$,
  Lemma~\ref{lem:conv_div}-{\it 2} gives us a geodesic arc $L_q$
  through $q$ such that, passing to a subsequence,
  $N_{\sigma(n)}|_{L_q}$ becomes horizontal and orthogonal to $L_q$.
  The vector $N_{\sigma(n)}(q)$ converges to a horizontal vector
  orthogonal simultaneously to $L$ and $L_q$, from which we deduce
  that $L_q\subset L$, and so $L_q\subset\Lambda$.

  Finally, we prove $\Lambda$ is a closed set, which finishes
  Lemma~\ref{lem:continuation}.  Let $\{q_m\}$ be a sequence of points
  in $\Lambda$ such that $q_m\to q\in L$. Let us prove that
  $q\in\Lambda$. For each $m$, there exists a subsequence of
  $\{u_n\}$ such that $N_n|_{[p,q_m]}$ becomes horizontal and
  orthogonal to $L$. A diagonal argument allows us to take a common
  subsequence of $\{u_n\}$ (also denoted by $\{u_n\}$) such that
  $N_n|_{[p,q_m]}$ becomes horizontal and orthogonal to $L$, for every
  $m$. For every $m$,  there is a geodesic arc $L_{q_m}\subset L$ centered
at $q_m$ satisfying Lemma~\ref{lem:conv_div}-{\it 2} whose length depends
only on ${\rm dist}_{\H^2}(q_m,\partial\Omega)$. Hence, $q\in L_{q_m}$
for any $m$ large enough, and so $q\in\Lambda$.
\end{proof}

\begin{proposition} \label{prop:conv} Suppose the divergence set of
  $\{u_n\}$ is a countable set of lines.  Then there exists a
  subsequence of $\{u_n\}$ (denoted as the original sequence) such
  that:
  \begin{enumerate}
  \item The divergence set ${\cal D}$ of
    $\{u_n\}$ is composed of a countable number of divergence
    lines, pairwise disjoint.
   %Suppose ${\cal B}\neq\emptyset$.
  \item For any component $\Omega'$ of ${\cal B}=\Omega-{\cal D}$ and
    any $p\in\Omega'$, $\{u_n-u_n(p)\}$ converges uniformly on compact
    sets of $\Omega'$ to a minimal graph over $\Omega'$.
  \end{enumerate}
\end{proposition}

\begin{proof}
  Suppose $L_1$ is a divergence line of $\{u_n\}$.
  Lemma~\ref{lem:conv_div} assures that, passing to a subsequence,
  $\{N_n(q)\}$ converges to a horizontal vector orthogonal to $L_1$ at
  $q$, for each $q\in L_1$. Observe that the divergence set associated
  to such a subsequence (denoted again by $\{u_n\}$) is contained in
  the divergence set of the original sequence. In particular, the
  divergence set for such a subsequence, denoted by ${\cal D}$,
  contains a countable number of divergence lines.

  Suppose there exists a divergence line $L_2\subset{\cal D}$,
  $L_2\neq L_1$. Passing to a subsequence, we obtain that
  $\{N_n(q)\}$ converges to a horizontal vector orthogonal to $L_2$,
  for each $q\in L_2$. In particular, $L_1\cap L_2=\emptyset$, since
  if there exists some $q\in L_1\cap L_2$ then $N_n(q)$ would
  converge to a horizontal vector orthogonal to both $L_1,L_2$
  simultaneously, a contradiction. The ``new'' divergence set ${\cal
    D}$ is then a countable set of divergence lines containing $L_1$ and
$L_2$, with $L_1\neq L_2$. 

  Continuing the above argument, we obtain with a diagonal process a
  subsequence of $\{u_n\}$ (also denoted by $\{u_n\}$) whose
  divergence set ${\cal D}$ is composed of a countable number of
  pairwise disjoint divergence lines $L_i$.

  Now consider a countable set of points $\{p_i\}_i$ dense in ${\cal
    B}$, the convergence domain associated to the subsequence obtained
  in the previous argument.  Using Lemma~\ref{lem:conv_div}-{\it 1}
  and a diagonal argument, we obtain a subsequence of $\{u_n\}$ such
  that $\{u_n-u_n(p)\}$ converges uniformly on compact sets of
  $\Omega'$ to a minimal graph, for every component $\Omega'$ of
  ${\cal B}$ and every $p\in \Omega'$. This finishes the proof of
  Proposition~\ref{prop:conv}.
\end{proof}

\begin{remark}
  In Proposition~\ref{prop:conv} we can remove the hypothesis ${\cal
     D}$ is a countable set of divergence lines, and we obtain that,
  after passing to a subsequence, ${\cal D}$ is composed of pairwise
   disjoint divergence lines and, up to a vertical translation, we have
  uniform convergence on compact sets of each component of the
   convergence domain ${\cal B}$. The proof of this fact is more
   involved and will be included in~\cite{cor3}. 

  We will only use Proposition~\ref{prop:conv} in the case the
  divergence set ${\cal D}$ is composed of a finite number of
  divergence lines.
\end{remark}

Let $\{u_n\}$ be a subsequence given by Proposition \ref{prop:conv}. We
consider $\Omega'$ a connected component of $\boB$. Its boundary is
composed of subarcs of $\partial\Omega$ and divergence lines. Let us
understand the limit $u$ of $\{u_n-u_n(p)\}$ in $\Omega'$
($p\in\Omega'$). Let $T$ be a subarc of $\partial\Omega'$ included in
a divergence line. From the convergence of $\{N_n\}$ along $T$,
$F_{u_n}(T)$ converges to $\pm|T|$. Since $|X_{u_n}|$ is bounded by
$1$, this implies that $F_u(T)=\pm|T|$. Then by Lemma \ref{lem:xulf}, $u$
takes value $\pm\infty$ on $T$. In fact we have a stronger result.

\begin{lemma}\label{lem:div}
Let $\{u_n\}$ be a sequence of minimal graphs on $\Omega$. We
assume that $\{u_n\}$ converges to a minimal graph $u$ on a connected
subdomain $\Omega'$ of $\Omega$. Let $T$ be a subarc in $\partial\Omega'$
included in a divergence line for the sequence $\{u_n\}$ such
that $X_{u_n}\rightarrow \nu$ along $T$ with $\nu$ the outgoing normal to
$\Omega'$. Then if $p\in\Omega'$ and $q\in T$ we have 
$$
\lim_{n\rightarrow+\infty} u_n(q)-u_n(p)=+\infty
$$
\end{lemma}
\begin{proof}
Since $X_{u_n}\rightarrow \nu$ on $T$, $F_{u_n}(T)$ converges to $|T|$.
Thus $u$ takes the value $+\infty$ on $T$. Let $p$ and $q$ be as in
the lemma and consider the disk model for $\H^2$ assuming that
$q$ is at the origin, $T$
is a subarc of $\{x=0\}$ and $\nu$ points to the half-plane $\{x\ge 0\}$.
Let us prove: 
\begin{equation*}\label{croissance}
\text{There is }\epsilon >0\text{ such that }\dfrac{\partial
u_n}{\partial x}> 0\text{ on }\{-\epsilon<x\le 0,y=0\}\text{ for large
}n.
\tag{$*$}
\end{equation*}
Since $u=+\infty$ on $T$ there is $\epsilon >0$ such that $\dfrac{\partial
u}{\partial x}\ge 1$ on $\{-\epsilon<x<0,y=0\}$. The convergence
$u_n\rightarrow u$ implies : for every $0<\eta<\epsilon$, $\dfrac{\partial
u_n}{\partial x}> 0$ on $\{-\epsilon<x<-\eta,y=0\}$ for large $n$. If
\eqref{croissance} is not true, considering
a subsequence if necessary, there is $q_n$ in
$\{-\epsilon<x\le 0,y=0\}$ with $\dfrac{\partial u_n}{\partial x}(q_n)= 0$.
Observe that it must be $q_n\rightarrow q$.

If the sequence $\{|\nabla u_n(q_n)|\}$ is bounded, $|\nabla u_n|$ is
uniformly bounded in a uniform disk around $q_n$. Since $q_n\rightarrow q$,
the sequence $\{|\nabla u_n(q)|\}$ is bounded which is false since $q$
lies on a divergence line. Hence, passing to a subsequence, we can assume
that $|\nabla
u_n(q_n)|\rightarrow +\infty$. Let $D_n^1$ be the $\delta$-geodesical
disk centered at $(q_n,0)$ in the graph of
$u_n-u_n(q_n)$ ($\delta$ is fixed small enough with respect to the
distance from $q$ to $\partial\Omega$). Since $\dfrac{\partial
u_n}{\partial x}(q_n)= 0$ we can prove as in Lemma \ref{lem:conv_div}
 that the sequence $\{D_n^1\}$ converges to the vertical disk in
$\{y=0\}\times\R$ centered at $(q,0)$ of radius $\delta$. Let $D_n^2$ be
the $\delta$-geodesical disk centered at $(q,0)$ in the graph of
$u_n-u_n(q)$. Since $T$ is part of a divergence line, $\{D_n^2\}$ converges
to the vertical disk in $\{x=0\}\times\R$ centered at $(q,0)$ or radius
$\delta$.
Because of both convergences, for large $n$, $D_n^1$ and $D_n^2$
intersect transversally. But this is impossible, since their normal vectors
at a point depends only on $\nabla u_n$.

Assertion \eqref{croissance} is then proved. Let $q_t$ be the point of
coordinates $(-t,0)$. Since $u$ takes the value $+\infty$ at $q$ we can
make $u(q_t)-u(p)$ as large as we want by taking $t$ small . Besides,
for large $n$, \eqref{croissance} gives $u_n(q)-u_n(p)\ge u_n(q_t)-u_n(p)$.
Since $u_n\rightarrow u$, we get $u_n(q)-u_n(p)\ge u(q_t)-u(p)-1$. This
proves the lemma.
\end{proof}

\begin{remark}\label{rem:plantas}
  Let $L$ be a divergence line and suppose there exist two components
  $\Omega_1,\Omega_2$ of ${\cal B}$ such that
  $L\subset\partial\Omega_i$, $i=1,2$. Consider points
  $p_1\in\Omega_1,\ p_2\in\Omega_2$. Passing to a subsequence,
  $\{u_n-u_n(p_i)\}$ converges uniformly on compact sets of
  $\Omega_i$ to a minimal graph $u_i:\Omega_i\to\R$. Assume
  $F_{u_1}(T)=|T|$ for each bounded arc $T\subset L$, when $L$ is
  oriented as $\partial\Omega_1$. Then $F_{u_2}(T)=-|T|$, when $L$ is
  oriented as $\partial\Omega_2$.  We deduce from Lemma~\ref{lem:div}
  that $\{(u_n-u_n(p_1))|_L\}$ diverges to $+\infty$ and
  $\{(u_n-u_n(p_2))|_L\}$ diverges to $-\infty$. In particular, we
  can deduce that $\{u_n-u_n(p_1)\}$ diverges uniformly on
  compact sets of $\Omega_2$ to $+\infty$. 
\end{remark}

Now, we are going to exclude the existence of some divergence lines
under additional constraints. In particular, if there exists minimal
graphs $w^+,w^-$ defined on a neighborhood ${\cal
  U}\subset\overline\Omega$ of a point $p\in\partial\Omega$ such that
$w^-\leq u_n\leq w^+$ for every $n$, then a divergence line cannot
arrive at $p$. We will state conditions for which such barriers exist.

\begin{proposition} \label{prop:lindiv} Let $\{u_n\}$ be the
  subsequence given by Proposition~\ref{prop:conv}. 
  \begin{enumerate}
  \item Let $C\subset\partial_\infty\Omega$ be a smooth arc where each
$u_n$
    extends continuously and suppose $\{u_n|_C\}$ converges to a
    continuous function $f$.  Then a divergence line $L_i$ cannot
    finish at an interior point of $C$.
  \item For every $n$, suppose there exists $M_n\geq 0$ such that
    $|u_n|\leq M_n$, and let $T\subset\partial\Omega$ be a bounded
    geodesic arc where $u_n$ extends continuously and $u_n|_T=M_n$ or
    $-M_n$. Then a divergence line cannot finish at an interior point
    of $T$.
\end{enumerate}
\end{proposition}

\begin{proof}
  Let $C\subset\partial_\infty\Omega$ be an arc as in item {\it 1}. 
Suppose
  $C$ is either an arc at $\partial_\infty\H^2$ or a strictly convex
  arc (with respect to $\Omega$). Let $p\in C$ and $C'$ be a
  neighborhood of $p$ in $C$ such that $\overline{C'}\subset C$.
  Consider the geodesic $\Gamma(C')\subset\H^2$ joining the endpoints
  of $C'$, and define the domain $\Delta\subset\H^2$ bounded by
  $C'\cup\overline{\Gamma(C')}$. For $C'$ small enough, we
  can assume $\Delta\subset\Omega$.

  Define $M=\max_{C'}|f|$. For $n$ big enough and $C'$ small enough,
  $|u_n|< M+1$ on $C'$, for every $n$.  Consider $w^+,w^-:\Delta\to\R$
  minimal graphs with boundary values
  \[
  \left\{\begin{array}{ll}
      w^+=M+1      & \mbox{, on } C'\\
      w^+=+\infty & \mbox{, on } \Gamma(C')
    \end{array}\right.
  \qquad \left\{\begin{array}{ll}
      w^-=-M-1      & \mbox{, on } C'\\
      w^-=-\infty & \mbox{, on } \Gamma(C')
    \end{array}\right.
  \]
(they exist by Lemma~\ref{lem:barrier} and Theorem~\ref{th1}, depending
on the case).  By the general maximum principle, $w^-\leq u_n\leq w^+$
  for every $n$.  Therefore, the Compactness Theorem says
  $\Delta\subset {\cal B}$, and so no divergence line finishes at $p$.

  Now suppose that $C$ is geodesic and $u_n|_C=c\in\R$ for every $n$.
  We can assume without loss of generality $c=0$. By reflecting the
  graph surface of $u_n$ about $C$, we obtain a minimal surface
  $\Sigma$ containing $C$, whose normal vector along $C$ is orthogonal
  to $C$. If there exists a divergence line $L$ with an endpoint at $p\in
C$, then we conclude $N_n(p)$ converges to a horizontal vector
  orthogonal to $L$. But this is impossible, since such a vector must
  be orthogonal to $C$. Hence, no divergence line finishes at $C$.

  Finally, suppose $C$ is geodesic and there exists a divergence line
  $L$ with endpoint $p\in C$. Fix $\ve>0$. Since $\{u_n|_C\}$
  converges to a continuous function $f$, there exists a small
  neighborhood $C'\subset C$ of $p$ such that $|u_n(q)-f(p)|<\ve$, for
  every $q\in C'$ and $n$ large enough. Consider a neighborhood ${\cal
    U}\subset\Omega\cup C$ of $p$ containing $C'$, and define
  $v_n:{\cal U}\to\R$ as the minimal graph with boundary values
  \[
  \left\{\begin{array}{ll}
      v_n=f(p) & \mbox{, in } C'\\
      v_n=u_n  & \mbox{, in } \partial{\cal U}-C'
    \end{array}\right.
  \]
  (it exists by Theorem~\ref{th1}). The general maximum principle for
  bounded domains assures
  \begin{equation}\label{eq:v_n}
  u_n-\ve\leq v_n\leq u_n+\ve.
  \end{equation}
  Next we prove that $L\cap{\cal U}$ is a divergence line for
  $\{v_n\}$, conveniently choosing $\ve$ and ${\cal U}$. Fix a point
  $q\in L\cap{\cal U}$. From the proof of Lemma~\ref{lem:conv_div}, we
  deduce there exists a neighborhood of $(q,0)$ in the graph  
$G(u_n-u_n(q))$ converging to the disk $D_L(q,\delta)\subset
  L\times\R$ of radius $\delta$ centered at $(q,0)$. Taking
  $\ve\leq\delta/2$, we conclude
  using (\ref{eq:v_n}) that a neighborhood of the point
  $(q,v_n(q)-u_n(q))$ in $G(v_n-u_n(q))$ converges to $D_L(q,\delta)$,
  and $L\cap{\cal U}$ is a divergence line for $\{v_n\}$
  (see~\cite{mazet0}, Proposition 1.4.8, for a detailed proof). But we
  know from the above argument this is not possible, as $v_n$ is
  constant on $C'$. This finishes item {\it 1}.

    Now, consider $T$ as in the hypothesis of {\it 2}, and let $p\in T$.
  Define $v_n=u_n-u_n(p)$ for every $n$. Clearly, $v_n|_T=0$ for every
  $n$. Then we obtain from item {\it 1} that a divergence line for
  $\{v_n\}$ cannot finish at $T$. Since the divergence lines
  associated to $\{u_n\}$ coincide with those of $\{v_n\}$, we
  have proved  Proposition~\ref{prop:lindiv}.
\end{proof}  

%%%%%%%%%%%%%%
\subsubsection{Solving the Jenkins-Serrin problem on unbounded domains}
\label{solutiondirichlet}
%%%%%%%%%%%%%%

   \begin{figure}\begin{center}
       \epsfysize=7cm \epsffile{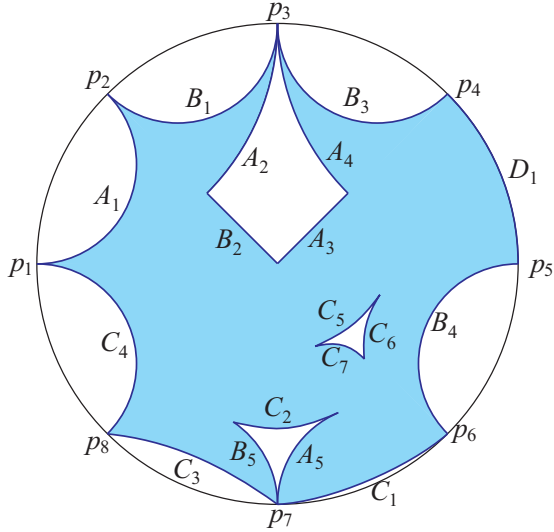}
     \end{center}
     \caption{An ideal Scherk domain.}
     \label{fig:general}
   \end{figure}

  Let $\Omega\subset\H^2$ be a domain whose boundary
$\partial_\infty\Omega$ consists of a
  finite number of geodesic arcs $A_i,B_i$, a finite number of
   convex arcs $C_i$ (convex towards $\Omega$) and a finite number
  of open arcs $D_i$ at $\partial_\infty\H^2$, together with their
  endpoints, which are called the vertices of $\Omega$ (see
Figure~\ref{fig:general}). We mark the $A_i$ edges by
  $+\infty$, the $B_i$ edges by $-\infty$, and assign arbitrary
  continuous data $f_i,g_i$ on the arcs $C_i,D_i$, respectively.
  Assume that no two $A_i$ edges and no two $B_i$ edges meet at a
  convex corner.  We will call such a domain $\Omega$ an {\it
    ideal Scherk domain}.

A polygonal domain $\mathcal{P}$ is said to be inscribed in $\Omega$ 
if ${\mathcal{P}}\subset\Omega$ and its vertices are
among the endpoints of the arcs $A_i,B_i,C_i$ and $D_i$; we notice that a
vertex may be in
$\partial_\infty\H^2$ and an edge may be one of the $A_i$ or $B_i$ (see
Figure~\ref{fig:horociclos}).

   For each ideal vertex $p_i$ of $\Omega$ at $\partial_\infty\H^2$, we
   consider a horocycle $H_i$ at $p_i$.  Assume $H_i$ is small enough
   so that it does not intersect bounded edges of $\partial\Omega$ and
   $H_i\cap H_j=\emptyset$ for every $i\neq j$. Given a polygonal
   domain ${\cal P}$ inscribed in $\Omega$, we denote by $\Gamma({\cal
P})$ the part of $\partial{\cal P}$ outside the horocycles, and (see
   Figure~\ref{fig:horociclos})
  \[
  \g=|\Gamma({\cal P})|, \qquad
  \a=\sum_i|A_i\cap\Gamma({\cal P})| \qquad {\rm and}\qquad
  \be=\sum_i|B_i\cap\Gamma({\cal P})| .
  \]
  \begin{figure}\begin{center}
      \epsfysize=7cm \epsffile{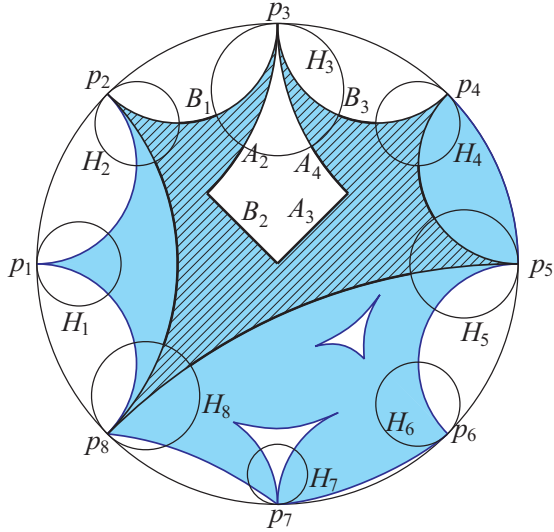}
    \end{center}
    \caption{An inscribed polygonal domain in $\Omega$}
    \label{fig:horociclos}
  \end{figure}

  \begin{theorem}
    \label{th2}
    If there is at least one edge $C_i$ or $D_i$ in
$\partial_\infty\Omega$,
    then a solution to the Dirichlet problem on $\Omega$ exists if and
    only if the horocycles $H_i$ can be chosen so that
    \begin{equation}
      \label{hipJS2} 
      2\a<\g\qquad\mbox{and }\qquad 2\be<\g
    \end{equation}
    for every polygonal domain ${\cal P}$ inscribed in $\Omega$.
  \end{theorem}

  \begin{remark}
    If these conditions hold for some choice of horocycles, then they
    also holds for all smaller horocycles.
  \end{remark}

  \begin{proof}
    Given a vertex $p_i\in\partial_\infty\H^2$ of $\Omega$, we
    consider a sequence of nested horocycles $\{H_{i,n}\}$
    converging to $p_i$. Assume $H_{i,n}\cap H_{j,n}=\emptyset$, for
    every $i\neq j$. Denote by ${\cal H}_{i,n}$ the horodisk bounded
    by $H_{i,n}$. Given an inscribed polygonal domain ${\cal
      P}\subset\Omega$, we call ${\cal P}_n$ the domain bounded by
    $\partial{\cal P}-\cup_i{\cal H}_{i,n}$ together with geodesic
    arcs contained in ${\cal P}\cap\left(\cup_i{\cal H}_{i,n}\right)$
    joining points in $\partial{\cal P}\cap
    \left(\cup_iH_{i,n}\right)$.  Define
    \[
    \g_n=|\partial{\cal P}-\cup_i{\cal H}_{i,n}|,\qquad
    \a_n=\sum_i|A_i\cap\partial{\cal P}_n|,\qquad
    \be_n=\sum_i|B_i\cap\partial{\cal P}_n|.
    \]
    Observe that both sequences $\{2\a_n-\g_n\}$ and
    $\{2\be_n-\g_n\}$ are monotonically decreasing.\\

    Let us first prove the conditions are necessary in Theorem~\ref{th2}.
    Assume there exists a solution $u$ to the Dirichlet problem on
    $\Omega$, and let ${\cal P}\subset\Omega$ be an inscribed polygon.
    Since either $\{C_i\}\neq\emptyset$ or $\{D_i\}\neq\emptyset$,
    there exists  a curve $\eta\subset\partial{\cal P}$ which
    is not an $A_i$ or $B_i$ edge.  Let $\widetilde\eta\subset\eta$ be
    a fixed bounded arc.  Lemma~\ref{lem:flux} assures
    $F_u(\partial{\cal P}_n)=0$, $\sum_iF_u(A_i\cap\partial{\cal
      P}_n)=\a_n$ and $|F_u(\partial{\cal
      P}_n\setminus (\cup_iA_i \cup\widetilde\eta))|\leq
\g_n-\a_n-|\widetilde\eta|$.
    Thus we obtain
    \[
    \a_n\leq \g_n-\a_n-|\widetilde\eta|+|F_u(\widetilde\eta)|+ \ve_n,
    \]
    where $\ve_n=|\partial{\cal P}_n-\partial{\cal P}|$.  This is,
    $2\a_n-\g_n<\ve_n-(|\widetilde\eta|-|F_u(\widetilde\eta)|)$.
    Analogously,
    \[
    2\be_n-\g_n<\ve_n-(|\widetilde\eta|-|F_u(\widetilde\eta)|).
    \]
    Since $|F_u(\widetilde\eta)|<|\widetilde\eta|$ (again by
    Lemma~\ref{lem:flux}) and $\ve_n$ converges to zero as $n$ goes to
    $+\infty$, then $\ve_n<(|\widetilde\eta| -F_u(\widetilde\eta))$
    for $n$ big enough. Therefore, condition (\ref{hipJS2}) is
    satisfied for ${\cal P}$ and the horocycles $H_{i,n}$, for $n$
    large enough.

    Finally, observe there are a finite number of inscribed polygonal
    domains ${\cal P}$ in $\Omega$ (there are a finite number of
    vertices of $\Omega$). Thus we can choose $H_i=H_{i,n}$ for $n$
    large so that (\ref{hipJS2}) is satisfied for any
    inscribed polygonal domain ${\cal P}\subset\Omega$.\\

    Let us now prove the conditions are sufficient.  We choose
    $H_{i,1}=H_i$. Thus we have $2\a_n<\g_n$ and $2\be_n<\g_n$ for every
    $n$.

    We now construct domains $\Omega_n$ converging to $\Omega$.  For
    any vertex $p_i\in\partial_\infty\H^2$ of $\Omega$, we consider a
    sequence of nested ideal geodesics $\G_{i,n}$ converging to
    $p_i$. By nested we mean that, if $\Delta_{i,n}$ is the component
    of $\H^2\backslash\G_{i,n}$ containing $p_i$ at its ideal boundary,
then
    $\Delta_{i,n+1}\subset\Delta_{i,n}$.  Assume $\G_{i,n}\cap
    \G_{j,n}=\emptyset$, for every $i\neq j$, and define
    \[
    A_{i,n}=A_i\setminus \cup_k\Delta_{k,n} ,\quad
B_{i,n}=B_i\setminus \cup_k\Delta_{k,n}\quad
    \mbox{ and}\quad C_{i,n}=C_i\setminus \cup_k\Delta_{k,n}.
    \]
    For $r>0$ big enough, the annulus bounded by $\partial_\infty\H^2$
    and the circle $\esf_{\H^2}(0,r)$ of radius $r$ (in the hyperbolic
    metric) centered at the origin of the Poincar\'e disk, does not
    intersect the bounded components of $\partial\Omega$. Consider a
    monotone increasing sequence of radii $\{r_n\}$ converging to
    $+\infty$. For $r_n$ big enough, we can assume
    $\esf_{\H^2}(0,r_n)$ intersects every geodesic $\G_{k,n}$ twice,
    and define by $D_{i,n}$ the component of
    $\esf_{\H^2}(0,r)\setminus\cup_k\Delta_{k,n}$ converging to $D_i$.  We
can
    naturally assign the values $g_i$ on each $D_{i,n}$.  Finally, let
    us call $\Omega_n$ the domain bounded by the edges
    $A_{i,n},B_{i,n},C_{i,n},D_{i,n}$, and the corresponding geodesic
    arcs $\G_{i,n}^j\subset\G_{i,n}$, together with their endpoints.

   \begin{figure}\begin{center}
       \epsfysize=7cm \epsffile{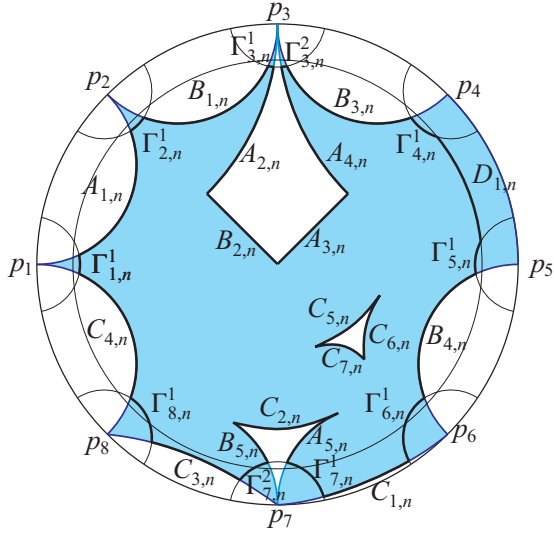}
     \end{center}
     \caption{Construction of the domain $\Omega_n$}
     \label{fig:general_dominio}
   \end{figure}

    Theorem~\ref{th1} assures, for each $m\in\N$, the existence
    of a unique minimal graph $u_m^n:\Omega_n\to\R$ with boundary
    values
    \[
    \left\{\begin{array}{ll}
        u_m^n=m      & \mbox{, on the } A_{i,n} \mbox{ edges.}\\
        u_m^n=-m      & \mbox{, on the } B_{i,n} \mbox{ edges.}\\
        u_m^n=f_{i,m} & \mbox{, on the } C_{i,n} \mbox{ edges.}\\
        u_m^n=g_{i,m} & \mbox{, on the } D_{i,n} \mbox{ edges.}\\
        u_m^n=0 & \mbox{, on the geodesic arcs } \Gamma_{i,n}^j.
      \end{array}\right.
    \]
    where $f_{i,m}$ (resp.  $g_{i,m}$) denotes the function $f_i$
    (resp. $g_i$) truncated above and below by $m$ and $-m$,
    respectively. By the maximum principle for bounded
    domains, $-m\leq u_m^n\leq m$, for every $n$. Then we can
    extract, by using the compactness theorem and a diagonal argument,
    a subsequence of $\{u_m^n\}_n$ converging uniformly on compact
    subsets of $\Omega$ to a minimal graph $u_m:\Omega\to[0,m]$ with
    boundary data
  \[
  \left\{\begin{array}{ll}
      u_m=m      & \mbox{, on the } A_i \mbox{ edges.}\\
      u_m=-m      & \mbox{, on the } B_i \mbox{ edges.}\\
      u_m=f_{i,m} & \mbox{, on the } C_i \mbox{ edges.}\\
      u_m=g_{i,m} & \mbox{, on the } D_i \mbox{ edges.}\\
    \end{array}\right.
  \]
  Such boundary data are obtained from a standard barrier argument,
  using as barriers the ones described in~\cite{cor2}.
  
  We are going to prove that a subsequence of $\{u_m\}$ converges to
  a solution to the Dirichlet problem on $\Omega$, proving
  Theorem~\ref{th2}.  We know from Proposition~\ref{prop:lindiv}
  that divergence lines for $\{u_m\}$ can only arrive at vertices of
  $\Omega$. In particular, there exists a finite number of divergence
  lines, and so ${\cal B}\neq\emptyset$.

  Passing to a subsequence, we can assume $\{u_n\}$ satisfies
  Proposition~\ref{prop:conv}.  Now suppose by contradiction that
  ${\cal B}\neq\Omega$; i.e., suppose there exists a divergence line
  $L\subset{\cal D}$.  We then deduce from Remark~\ref{rem:plantas}
  there exists a component ${\cal P}\subset{\cal B}$ such that
  $\{u_n\}$ diverges uniformly on compact sets of ${\cal P}$, say to
  $+\infty$ (the case $-\infty$ follows similarly).  Take a point
  $p\in{\cal P}$. Then $\{u_n-u_n(p)\}$ converges uniformly on
  compact subsets of ${\cal P}$ to a minimal graph $u:{\cal P}\to\R$.
  Observe that $u$ diverges to $-\infty$ as we approach any edge in
  $\partial{\cal P}\cap\left(\partial\Omega-\cup_i A_i\right)$ within
  ${\cal P}$. We then get ${\cal P}$ is a polygonal domain and
  $F_u(T)=-|T|$ for every bounded arc $T\subset\partial{\cal
    P}\cap\left(\partial\Omega-\cup_i A_i\right)$.

  \begin{claim}\label{cl:P}
    We can choose the polygonal domain ${\cal P}\subset{\cal B}$ so
    that $F_u(T)= -|T|$ for any bounded geodesic arc
    $T\subset\partial{\cal P}-\cup_iA_i$.
  \end{claim}

  Assume Claim~\ref{cl:P} is true and define ${\cal P}_n$ as at the
  beginning of the proof.  Thus $F_u(\partial{\cal P}_n-\cup_i
  A_i-(\partial{\cal P}_n-\partial{\cal P}))=-|\partial{\cal
    P}_n-\cup_i A_i-(\partial{\cal P}_n-\partial{\cal P})|$. By
  Lemma~\ref{lem:flux},
  \[
  \left\{\begin{array}{l} \sum_iF_u(A_i\cap\partial{\cal P}_n)+
      F_u(\partial{\cal P}_n-\partial{\cal P})\\
      \mbox{}\\
      \mbox{}\hspace{2cm} +F_u(\partial{\cal P}_n-\cup_i
      A_i-(\partial{\cal
        P}_n-\partial{\cal P}))=0,\\
      \mbox{}\\
      \left|\sum_iF_u(A_i\cap\partial{\cal P}_n)+F_u(\partial{\cal P}_n-\partial{\cal P})\right|
      \leq\a_n+\ve_n,
    \end{array}\right.
  \]
  where $\ve_n=|\partial{\cal P}_n-\partial{\cal P}|$, which converges
  to zero as $n\to+\infty$.
  Hence,
  \[
  \g_n-\a_n-\ve_n\leq\a_n+\ve_n.
  \]
  Thus we obtain $-2\varepsilon_n\leq 2\alpha_n-\gamma_n\leq
  2\alpha_1-\gamma_1$, for every $n$.  Since $\varepsilon_n\to 0$ as
  $n\to+\infty$, we obtain a contradiction to the first condition in
  (\ref{hipJS2}).  (If we suppose there exists a component ${\cal
    P}\subset{\cal B}$ such that $\{u_n\}$ diverges uniformly to
  $-\infty$ on compact sets of ${\cal P}$, we similarly achieve a
  contradiction using that $2\be_1-\gamma_1<0$).  Hence there are no
  divergence lines for $\{u_n\}$, and so ${\cal B}=\Omega$.

  Applying a flux argument as above, we obtain that $\{u_n\}$
  converges uniformly on compact sets of $\Omega$ to a minimal graph
  $u:\Omega\to\R$. Finally, using barrier functions as in~\cite{cor2}
  or those defined in Lemma~\ref{lem:barrier} for the $D_i$ edges, we
deduce
  that $u$ takes the desired boundary values, and this proves
  Theorem~\ref{th2}.

\bigskip

  So it only remains to prove Claim~\ref{cl:P}. Note we must only
  prove there exists a component ${\cal P}$ of ${\cal B}$ such that
  $\{u_n\}$ diverges to $+\infty$ uniformly on compact sets of
  ${\cal P}$ and $F_u(T)= -|T|$ for any bounded geodesic arc $T$
  contained in a divergence line in $\partial{\cal P}$.  Observe that,
  since ${\cal B}\neq\Omega$ is assumed, every component of ${\cal B}$
  contains at least one divergence line in its boundary.
  
  We know there exists a component ${\cal U}_0\subset{\cal B}$ which
  is an inscribed polygonal domain and such that $\{u_n\}$ diverges
  to $+\infty$ uniformly on compact sets of ${\cal U}_0$.  If ${\cal
    U}_0$ satisfies Claim~\ref{cl:P}, we have finished.  Otherwise,
  there exists a divergence line $L_0\subset\partial{\cal U}_0$ such
  that $F_{u_n}(L_0)\to |L_0|$ with the orientation induced by
  $\partial{\cal U}_0$.  Let ${\cal U}_1$ be the component of ${\cal
    B}$ different from ${\cal U}_0$ containing $L_0$ in its boundary.
  Hence $F_{u_n}(L_0)\to -|L_0|$ when $L_0$ is oriented as
  $\partial{\cal U}_1$. We deduce from Remark~\ref{rem:plantas} that
  $\{u_n\}$ diverges to $+\infty$ uniformly on compact sets of
  ${\cal U}_1$.

  If ${\cal U}_1$ satisfies the conditions of Claim~\ref{cl:P}, we are
  done. Otherwise, there exists another divergence line
  $L_1\subset\partial{\cal U}_1$ such that $F_{u_n}(L_1)\to|L_1|$ when
  $L_1$ is oriented as $\partial{\cal U}_1$.  We deduce from
  Lemma~\ref{lem:div} that, if $p_0\in {\cal U}_0$, then
  $\{u_n-u_n(p_0)\}$ diverges to $+\infty$ uniformly on compact sets
  of ${\cal U}_1$ and $(u_n-u_n(p_0))_{L_1}\to+\infty$. In particular,
  $L_1$ cannot be in $\partial{\cal U}_0$ because then
  $F_{u_n}(L_1)\to-|L_1|$, with the orientation in $L_1$ induced by
  $\partial{\cal U}_0$, in contradiction with
  $(u_n-u_n(p_0))_{L_1}\to+\infty$. Then there exists a component
  ${\cal U}_2$ of ${\cal B}$ different from ${\cal U}_0,{\cal U}_1$
  containing $L_1$ in its boundary.

  Since there are a finite number of components of ${\cal B}$, we
eventually obtain a component ${\cal U}_k$ of ${\cal B}$
  satisfying Claim~\ref{cl:P}.  This completes the proof of
  Theorem~\ref{th2}.
\end{proof}

\begin{theorem}
  \label{th3}
  Suppose that both families $\{C_i\}_i$ and $\{D_i\}_i$ are empty.
  Then, there exists a solution to the Dirichlet problem on $\Omega$ if and
  only if we can choose the horocycles $H_i$ so
  that $\alpha_1=\beta_1$ when ${\cal P}=\Omega$, and
  \[
  2\a_1<\g_1\qquad\mbox{and }\qquad 2\be_1<\g_1
  \]
  for all others polygonal domain ${\cal P}$ inscribed in $\Omega$.
  Moreover, the solution is unique up to translation, if it exists.
\end{theorem}

\begin{proof}
  Note that $\alpha_n-\beta_n$ does not depend on $n$.

  The proof of this theorem follows exactly as in the fourth case of
  the proof of Theorem~\ref{th1}. We must only clarify some points:
  \begin{enumerate}
  \item Now it is not straightforward to obtain $E_c=\cup_i E_c^i$
    and $F_c=\cup_j F_c^j$. A detailed proof can be found
    in~\cite{cor2}.
  \item Once we have the minimal graph $u:\Omega\to\R$ obtained as the
    limit of a subsequence of $\{u_n\}$, we must verify it satisfies
    the desired boundary conditions; this is, we must prove that both
    sequences $\{\mu_n\}$ and $\{n-\mu_n\}$ diverge as $n\to+\infty$.
  \end{enumerate}

  Suppose $\mu_n\to\mu_\infty<+\infty$ as $n\to+\infty$. Hence,
  $u=-\mu_\infty$ on each $B_i$ edge and $u$ diverges to $+\infty$
  when we approach $A_i$ within $\Omega$. From Lemma~\ref{lem:flux}, we get:
  \begin{itemize}
  \item $ \sum_i F_u(A_{i,n})+\sum_i F_u(B_{i,n})+\sum_{i,j}
    F_u(\Gamma_{i,n}^j)=0$,
  \item $\sum_i F_u(A_{i,n})=\alpha_n$,
  \item $\sum_i F_u(B_{i,1})<\beta_1$, so there exists $\delta>0$ such
    that $\sum_i F_u(B_{i,1})\leq\beta_1-\delta$. Then
    $F_u(B_{i,n})=F_u(B_{i,1})+F_u(B_{i,n}-B_{i,1})<\beta_n-\delta$,
    for every $n$.
  \item $\sum_{i,j} F_u(\Gamma_{i,n}^j)< \varepsilon_n$, where
    $\varepsilon_n=\sum_{i,j} |\Gamma_{i,n}^j|$.
  \end{itemize}
  Hence $\alpha_n-\beta_n<\varepsilon_n-\delta$, for every $n$. Since
  $\varepsilon_n\to 0$ as $n\to+\infty$, we obtain
  $\alpha_n-\beta_n<0$ for $n$ large enough, a contradiction.
  Analogously, we obtain $n-\mu_n\to+\infty$ as $n\to+\infty$. The
Uniqueness part follows from Theorem~\ref{th:max2}, and Theorem~\ref{th3}
is proved.
\end{proof}

%%%%%%%%%%%%%%%%%%%

\subsection{A minimal graph in $\H^2\times\R$ with non-zero flux}
\label{zeroflux}

%%%%%%%%%%%%%%%%%%%

Let $\Omega\subset\H^2$ be an unbounded domain whose boundary consists
of two components:
\begin{itemize}
\item  $\Gamma_{\rm ext}=$ an outer component composed of consecutive open
  ideal geodesics $A_1,B_1,\cdots,A_k,B_{k_1}$ sharing their endpoints
  at infinity.
\item $\Gamma_{\rm int}=$ an interior component consisting of open convex
  (convex towards $\Omega$) arcs $C_1,\cdots,C_{k_2}$, together with
  their endpoints.
\end{itemize}

Take a domain $\Omega$ as above satisfying (\ref{hipJS2}) for every
inscribed polygonal domain ${\cal P}$ and such that $\alpha_1>\beta_1$
when ${\cal P}=\Omega$. For example, consider a small deformation (as
in Figure~\ref{fig:new}) of a domain $\Omega'$ whose inner boundary is
composed of convex arcs together with their endpoints, and its outer
boundary consists of an ideal polygonal curve with vertices on the
$2k$-roots of $1$ (in the picture, $k=4$).

\begin{figure}\begin{center}
\epsfysize=6cm \epsffile{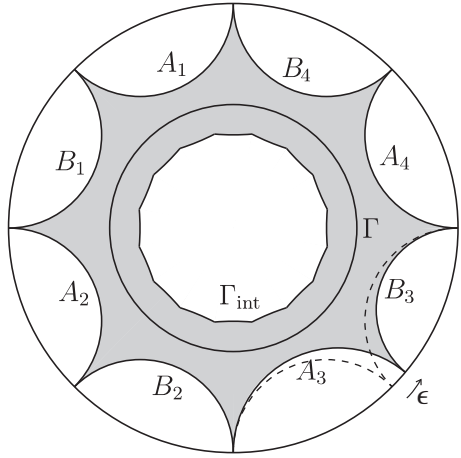}
\end{center}
\caption{The shadowed region is one of the domains considered in
Section~\ref{zeroflux}}
\label{fig:new}
\end{figure}

By Theorem~\ref{th2}, there exists a minimal graph $u:\Omega\to\R$
which takes boundary values $+\infty$ on the $A_i$ edges, $-\infty$ on
the $B_i$ edges, and $0$ on the $C_i$ edges.  Let
$\Gamma\subset\Omega$ be a curve homologous to $\Gamma_{\rm int}$.
Hence,
\[
F_u(\Gamma)=\sum_i F_u(A_{i,n})+\sum_i F_u(B_{i,n})+\sum_i
F_u(\Gamma_{i,n})
\]
\[
=\alpha_n-\beta_n+\sum_i F_u(\Gamma_{i,n}),
\]
where $\alpha_n=\sum_i |A_{i,n}|$ and $\beta_n=\sum_i
|B_{i,n}|$. Since $\alpha_n-\beta_n$ does not depend on $n$, we obtain
\[
|F_u(\Gamma)-\alpha_1+\beta_1|\leq \sum_i |F_u(\Gamma_{i,n})|\leq
\sum_i |\Gamma_{i,n}|.
\]
Finally, we know that $\sum_i |\Gamma_{i,n}|\to 0$, so $F_u(\Gamma)=\alpha_1-\beta_1>0$.

%%%%%%%%%%%%%%%%%%%

\subsection{The uniqueness problem in  $\H^2\times\R$}
\label{secPrincMax}

%%%%%%%%%%%%%%%%%%%

In this section we study the uniqueness of solutions constructed in
Theorems \ref{th2} and \ref{th3}. In the first subsection, we give a
maximum principle for solutions of the Dirichlet problem under some
constraints. In the second,
we construct a counterexample to a general uniqueness result.

%%%%%%%%%%%%%%
\subsubsection{Maximum principle}
%%%%%%%%%%%%%%

Maximum principles for unbounded domains in $\H^2$ are already known in
special cases. For example, the proof  of Collin and Rosenberg for the
maximum principle in~\cite{cor2} admits the following
generalization.

\begin{theorem}[\cite{cor2}]\label{th:max2}
Let $\Omega\subset\H^2$ be a domain (not necessarily simply connected)
whose boundary is composed of a finite number of convex arcs together with
their endpoints, possibly at infinity. Assume the following condition (C-R)
holds. Consider a domain ${\cal O}\subset\Omega$ and two minimal
graphs $u_1,u_2$ on ${\cal O}$ which extend continuously to $\overline{\cal
O}$. If $u_1\leq u_2$ on $\partial{\cal O}$, then $u_1\leq u_2$ in ${\cal
O}$.
\end{theorem}

The aim of this section is to prove that we can weaken the hypothesis
on the asymptotic behaviour of $\Omega$ when some constraints are
satisfied by the boundary data. Before stating our result, we need to
introduce some definitions. We notice that some notations for domains
we consider are different from the ones in
Subsection~\ref{solutiondirichlet}.

We consider domains $\Omega\subset\H^2$ whose boundary
$\partial_\infty\Omega$ is composed of a
finite number of open arcs $C_i$ in $\H^2$ and arcs $D_i$ in
$\partial_\infty\H^2$ together with their endpoints (the $C_i$ are not
supposed to be convex). The endpoints of the arcs $C_i$ and $D_i$ are
called vertices of $\Omega$ and those in $\partial_\infty\H^2$ are
called ideal vertices of $\Omega$. Let $p$ be an ideal
vertex of $\Omega$ and $\Gamma_1$ and $\Gamma_2$ be two adjacent boundary
arcs at $p$. Let $(\phi,\theta)$ be polar coordinates centered at $p$.
Consider a parametrization of $\Gamma_i$, $\gamma_i:
[0,1]\rightarrow \overline{\{\phi\le 0\}}^\infty$,
with $\gamma_i(0)=p$ and $\gamma_i(1)\in\{\phi=0\}$. We denote the
polar coordinates of the parametrization by $\gamma_i(t)=
(\phi_i(t),\theta_i(t))$ and assume that $\theta_1(1)\le \theta_2(1)$.

\begin{definition}
We say that $\Omega$ \emph{has necks near} $p$ if
$$
\liminf_{\substack{q\in\Gamma_1\\ q\rightarrow p}}d(q,\Gamma_2)=
\liminf_{\substack{q\in\Gamma_2\\ q\rightarrow p}}d(q,\Gamma_1)=0
$$
and the domain $\Omega$ is called \emph{admissible} if, for every
ideal vertex $p$ of $\Omega$, we have one of the following situations:
\begin{itemize}
\item[\textbf{type 1}] $\Omega$ has necks near $p$ or
\item[\textbf{type 2}] $\dis\liminf_{t\rightarrow 0}\theta_2(t)>0$ and
$\dis \limsup_{t\rightarrow 0}\theta_1(t)<\pi$.
\end{itemize}
\end{definition}
The limits of the second type do not depend on the choice of polar
coordinates. We notice that, if all $C_i$ are convex arcs (as in
section~\ref{solutiondirichlet}), every ideal vertex is of second type
\textit{i.e.} $\Omega$ is admissible. The hypothesis type $2$ means that
the adjacent arcs do not arrives ``tangentially'' to
$\partial_\infty\H^2$ on the same side of $p$. As in Figure
\ref{type1and2}, consider an ideal vertex $p$ such that, near $p$, $\Omega$
is the domain between to horocycles $p$. The distance between $\Gamma_1$
and $\Gamma_2$ is constant so $p$ is not a type 1 vertex. Besides we have
$\lim_{t\rightarrow 0} \theta_2(t)=0$, thus $p$ is
not a type 2 vertex. This is the kind of situation that we avoid by our
definition.
\begin{figure}[h]
\begin{center}
\resizebox{0.6\linewidth}{!}{\input{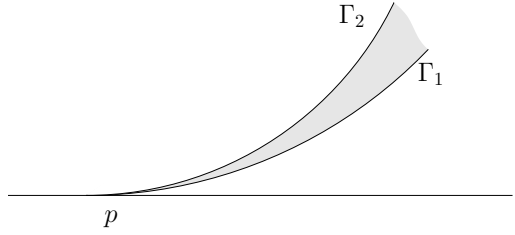}}
\caption{An ideal vertex which is neither type 1 nor type
2}\label{type1and2}
\end{center}
\end{figure}

Let $p$ be an ideal vertex of an admissible domain $\Omega$. \textit{A
priori}, this point is the endpoint of $2n$ arcs $\Gamma_i$ in
$\partial_\infty\Omega$ (see Figure~\ref{courb-multipl}). As above, let
$\gamma_i:
[0,1]\rightarrow \overline{\{\phi\le
0\}}^\infty\subset\H^2\cup\partial_\infty\H^2$,
$\gamma_i(t)=(\phi_i(t),\theta_i(t))$, be a
parametrization of $\Gamma_i$, with $\gamma_i(0)=p$ and
$\gamma_i(1)\in\{\phi=0\}$. We assume that $\theta_i(1)<
\theta_j(1)$ if $i< j$. Thus $\Omega\cap\{\phi\le 0\}$ is included in the
$n$ connected components of $\{\phi\le
0\}\backslash(\cup_i\Gamma_i)$ between $\Gamma_{2k-1}$ and
$\Gamma_{2k}$, for $k=1,\cdots,n$. When $u$ is a minimal graph on
$\Omega$ the
study of $u$ on the part between $\Gamma_{2k-1}$ and $\Gamma_{2k}$ depends
only on the values of $u$ on $\Gamma_{2k-1}$, $\Gamma_{2k}$ and the other
boundary arcs of $\Omega\cap\{\phi\le 0\}$ between $\Gamma_{2k-1}$ and
$\Gamma_{2k+1}$. Thus the study on each part will be done separately; so we
can assume that each ideal vertex is the endpoint of only two arcs in
$\partial_\infty\Omega$.
\begin{figure}[h]
\begin{center}
\resizebox{0.6\linewidth}{!}{\input{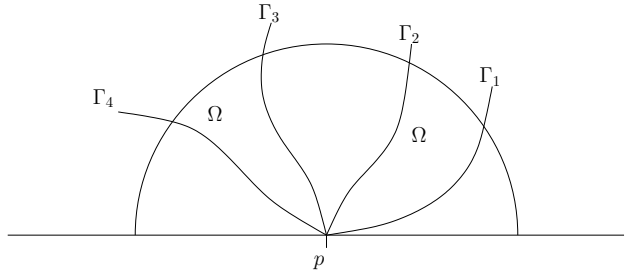}}
\caption{An ideal vertex with more than two adjacent
boundary arcs}\label{courb-multipl}
\end{center}
\end{figure}

Let $u$ be a minimal graph on an admissible domain $\Omega$. We say that
$u$ is \emph{admissible or an admissible solution} if 
\begin{itemize}
\item $u$ extends continuously to $\cup_iD_i$,
\item $u$ tends to $+\infty$ on $A(u)\subset\partial\Omega$ with
$A(u)$ is a finite union of open subarcs of $\cup_i C_i$,
\item $u$ tends to $-\infty$ on $B(u)\subset\partial\Omega$ with
$B(u)$ is a finite union of open subarcs of $\cup_i C_i$ and 
\item $u$ extends continuously to $\cup_iC_i \backslash
\overline{A(u)\cup B(u)}$.
\end{itemize}
We remark that each connected component of $A(u)$ and $B(u)$ is a
geodesic arc (see Theorem 10.4 in \cite{os1} for the Euclidean case and
Lemma \ref{lem:convexhull}). Also, we do not say anything about
the values of $u$ at the vertices of $\Omega$ and the endpoints of $A(u)$
and $B(u)$. Thus, in the following, the hypotheses on the boundary values
of an admissible solution $u$ will be only made where it is well defined
\textit{i.e.} $\cup_i D_i$, $A(u)$, $B(u)$ and $\cup_iC_i \backslash
\overline{A(u)\cup B(u)}$. As an example, in Theorem \ref{maxprinciple},
we shall write $u_2\le u_1$ on
$\partial_\infty\Omega$, this means that, $A(u_2)\subset A(u_1)$,
$B(u_1)\subset B(u_2)$ and $(\cup_i D_i)\bigcup(\cup_i C_i\backslash
\overline{A(u_2)\cup B(u_1)}$ is non empty and $u_2\le u_1$ on it (on
$A(u_1)\backslash\overline{A(u_2)}$ and $B(u_2)\backslash\overline{B(u_1)}$
the inequality is automatically satisfied). When $(\cup_i
D_i)\bigcup(\cup_i C_i\backslash \overline{A(u_2)\cup B(u_1)}$ is empty
then $u_1$ and $u_2$ are solutions of the Dirichlet problem studied in
Theorem \ref{th3} and we already know that $u_1-u_2$ is constant so no new
theorem is needed. Let us now state our generalization of
Theorem~\ref{th:max2}.

\begin{theorem}[General maximum principle]\label{maxprinciple}
Let $\Omega\subset\H^2$ be an admissible domain and $u_1$ and $u_2$ be two
admissible solutions. We assume that $u_2\le u_1$ on $\partial_\infty
\Omega$. Also we assume that the behaviour near each ideal vertex
$p\in\partial_\infty\H^2$ is one of the following:
\begin{itemize}
\item[\textbf{type 1}] $\Omega$ has necks near $p$,
\item[\textbf{type 2-i}] $\liminf_p u_1+\ve>\limsup_p u_2$ (for
every $\ve>0$) along both boundary components with $p$ as endpoint,
\item[\textbf{type 2-ii}] if $A\subset A(u_2)\subset A(u_1)$ (resp.
$B\subset B(u_1)\subset B(u_2)$) is a geodesic arc with $p$ as endpoint
and $\Gamma$ is the other boundary arc in $\partial_\infty\Omega$ with
endpoint $p$, %that bounds $\Omega$ near $p$, we have 
$\liminf_p u_1+\ve>\limsup_p u_2$ (for every $\ve>0$) along $\Gamma$. 
\end{itemize}
Then we have $u_2\le u_1$ in $\Omega$.
\end{theorem}

Let us make some comments on the hypotheses of the theorem. First the
hypothesis (C-R) made by Collin and Rosenberg in Theorem~\ref{th:max2}
implies that, near each ideal vertex, $\Omega$ has necks.
Thus Theorem \ref{maxprinciple} generalizes Theorem~\ref{th:max2}. We
notice that,
when a vertex $p$ is the endpoint of two geodesic arcs (for example, 
one in $A(u_2)$ and the other in $B(u_1)$), $\Omega$ has necks near
$p$. Moreover, the hypothesis $\liminf_p u_1+\ve>\limsup_p u_2$ along
a boundary component which has $p$ as endpoint means that we are in one of
the following three cases:
\begin{align}
&\liminf_p u_1=+\infty\textrm{ and }\limsup_p u_2<+\infty, \\
&\liminf_p u_1>-\infty\textrm{ and }\limsup_p u_2=-\infty, \\
&-\infty<\limsup_p u_2\le\liminf_p u_1<+\infty.\label{case3}
\end{align}
in the third case, the boundary data for $u_1$ and $u_2$ ``stay close'' so
it is the more complicated case. Hence the proof will be written in
this case; small changes suffice to treat the first two cases. We remark
that our theorem does not deal with the case $\lim_p u_1=\lim_p
u_2=+\infty$.

The proof of Theorem~\ref{maxprinciple} is long and needs some preliminary
results that may have their own interest.

Let $\Omega$ be a domain in $\H^2$, we say that $\Omega$ \emph{has a finite
number of point-ends} if there exist $p_1,\cdots, p_n\in
\partial_\infty\H^2$ and $(\phi_i,\theta_i)$ polar coordinates centered at
$p_i$ such that:
$$
\textrm{for every }m<0 \textrm{ and }i,\ \Omega\cap
\cup_i\{\phi_i>m\}\textrm{ is
compact and }\Omega\cap \{\phi_i<m\}\neq\emptyset.
$$
The $p_i$ are the point-ends (we do not assume anything about the
connectedness of $\Omega\cap \{\phi_i<m\}$). We say the
point-end $p_i$ \emph{is in a corridor} if there exists
$\alpha\in(0,\pi/2)$ and $m<0$ such
that:
$$
\Omega\cap \{\phi_i<m\}\subset \{\alpha<\theta_i<\pi-\alpha\}
$$
We notice that these definitions do not depend on the choice of
$(\phi_i,\theta_i)$.

Let $\Omega\subset\H^2$ be an admissible domain and $u_1$ and $u_2$ be two
admissible solutions on $\Omega$. We assume that $u_1 \ge u_2$ on
$\partial_\infty\Omega$. Let $\ve$ be positive with $O=\{u_1\le u_2-\ve\}$
nonempty. Since $u_1\ge u_2$ on the $D_i$, $O$ has a finite number of
point-ends that are among the ideal vertices of $\Omega$. With this
setting, we have a first result which follows the ideas of Collin and Krust
in \cite{ck1}.

\begin{proposition}\label{colkru}
Let $\Omega\subset \H^2$, $u_1$, $u_2$ admissible solutions on $\Omega$,
$\ve>0$ and $O$ be as above. The subset $O$ is assumed to be nonempty and,
for each point-end $p$, we assume  that either $p$ is in a corridor or
$\Omega$ has necks near $p$. Then the function $u_1-u_2$ is not bounded
below.
\end{proposition}

\begin{proof}
First, we can assume that $\ve$ is a regular value of $u_2-u_1$ and so
$\partial O\cap\Omega$ is smooth. Let us assume that the proposition is not
satisfied
\textit{i.e.} there exists $M>0$ such that $u_2-u_1\le M$. 

Let $K$ be a domain in $\H^2$ with smooth boundary such that
$\overline{\Omega\cap K}$ is compact. We notice that $\partial
O\cap (\cup_i D_i)=\emptyset$ and $\partial O\cap (\cup_i C_i)\subset
\overline{A(u_2)\cup B(u_1)}$. For $\delta>0$ small, we denote by
$N_\delta$ the closed $\delta$-neighborhood of $\overline{A(u_2)\cup
B(u_1)}$ and define:
$$
O(K,\delta)=\big(O\cap K\big)\backslash N_\delta 
$$
We notice that $\partial O(K,\delta)$ is piecewise smooth and is
included in $\Omega$. This boundary can be decomposed in three parts:
\begin{itemize}
\item $\partial_1(K,\delta)=\partial O(K,\delta)\cap
\partial O$ on which $u_2-u_1=\ve$,
\item $\partial_2(K,\delta)=\partial O(K,\delta)\cap
\partial N_\delta$,
\item $\partial_3(K,\delta)=\partial O(K,\delta)\cap (\partial
K\backslash \partial O)$.
\end{itemize}

\begin{figure}[h]
\begin{center}
\resizebox{0.6\linewidth}{!}{\input{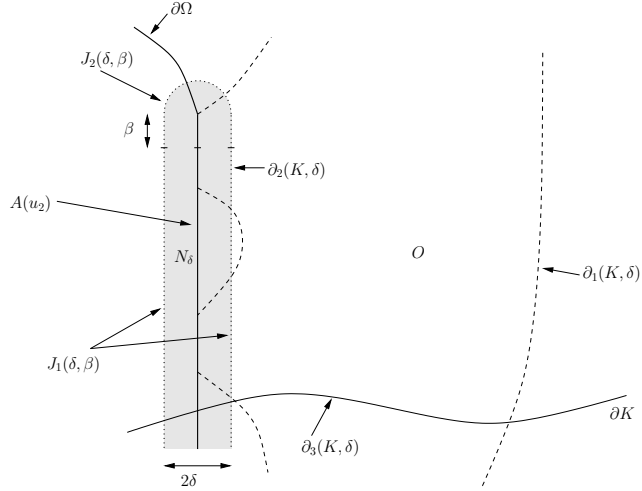}}
\caption{The boundary parts of $O(K,\delta)$}
\end{center}
\end{figure}

Let us define $u=u_2-u_1-\ve$, $X=X_{u_2}-X_{u_1}$ and $\nu$ the outgoing
normal from $O(K,\delta)$. Let us prove that:
\begin{equation}\label{lim-intdpsi}
\lim_{\delta\rightarrow0}\left|\int_{\partial_2(K,\delta)} u\langle
X,\nu\rangle\right| =0
\end{equation}
Since $\displaystyle \left|\int_{\partial_2(K,\delta)} u \langle
X,\nu\rangle\right|\le M\int_{\partial_2(K,\delta)}  \left|\langle
X,\nu\rangle\right|$, it suffices to prove

\begin{claim}\label{limint}
we have:
$$
\lim_{\delta\rightarrow 0}\int_{\partial_2(K,\delta)}
\left|\langle X,\nu\rangle \right|
=0
$$
\end{claim}

The connected components of $A(u_2)\cup B(u_1)$ are geodesic arcs. In
such a component, for $\beta>0$, a subarc is composed of points at a
distance larger than $\beta$ from the endpoints. We denote
by $I(\beta)$ the union of all these subarcs. Now, in $\partial
N_\delta$, some points are at distance $\delta$ from  $I(\beta)$ (we
denote this part $J_1(\delta,\beta)$) and the other
points are at distance $\delta$ from 
$\overline{A(u_2)\cup B(u_1)}\backslash I(\beta)$ (we denote this
part $J_2(\delta,\beta)$). We notice that the length of
$J_2(\delta,\beta)$ is bounded and 
$$
\lim_{\delta\rightarrow 0}{\ell}(J_2(\delta,\beta))=2n_0\beta
$$
where $n_0$ is the number of endpoints of $A(u_2)\cup B(u_1)$ in $\H^2$. We
have:
\begin{align*}
\int_{\partial_2(K,\delta)} \left|\langle X,\nu\rangle \right|&=
\int_{J_1(\delta,\beta)\cap\partial O(K,\delta)}
\left|\langle X,\nu\rangle \right|+ 
\int_{J_2(\delta,\beta)\cap\partial O(K,\delta)}
\left|\langle X,\nu\rangle \right|\\
&\le \int_{J_1(\delta,\beta)\cap\partial O(K,\delta)}
|X|+2\ell(J_2(\delta,\beta))\\
&\le \ell(J_1(\delta,\beta)\cap\partial O(K,\delta))
\max_{J_1(\delta,\beta)\cap \partial O(K,\delta)} |X|+
2\ell(J_2(\delta, \beta))
\end{align*}
As $\delta$ goes to $0$, $\max_{J_1(\delta,\beta)\cap\partial
O(K,\delta)}|X|$ tends to $0$ and $\ell(J_1(\delta,\beta)
\cap\partial O(K,\delta))$ is bounded (since $\Omega\cap K$ is compact).
Hence for every small $\mu>0$, we can take $\beta$ and $\delta$ small
enough such that:
$$
\int_{\partial_2(K,\delta)} \left|\langle X,\nu\rangle \right|\le \mu
$$
Claim~\ref{limint} is proved.

Also we have (see Lemma 1 in \cite{ck1} for the first inequality).
\begin{align*}
\iint_{O(K,\delta)}|X|^2\le \int_{\partial
O(K,\delta)}u \langle X,\nu\rangle &=\int_{\partial_1(K,\delta)}u
\langle X,\nu\rangle +
\int_{\partial_2(K,\delta)}u \langle X,\nu\rangle+
\int_{\partial_3(K,\delta)}u \langle X,\nu\rangle\\
&=\int_{\partial_2(K,\delta)}u \langle X,\nu\rangle +
\int_{\partial_3(K,\delta)}u \langle X,\nu\rangle
\end{align*}
We notice that $|X|^2\ge 0$ and $\int_{\partial_3(K,\delta)}
u|\langle X,\nu\rangle|\le 2M\ell(\partial_3(K,\delta))\le
2M\ell(\partial_3(K,0))$.
By \eqref{lim-intdpsi}, taking $\delta\rightarrow 0$ in
the above inequality, we get
\begin{equation}\label{inegalite}
\iint_{O(K,0)}|X|^2\le
\int_{\partial_3(K,0)}u\langle X,\nu\rangle
\end{equation}

Let $p_1,\cdots, p_n$ be the point-ends of $O$; they are numbered such
that $p_1,\cdots, p_k$ are in a corridor and $\Omega$ has necks near
$p_{k+1},\cdots,p_n$. For each $i$ we consider polar coordinates $(\phi_i,
\theta_i)$ centered at $p_i$, chosen such that the hyperbolic
half-planes $\{\phi_i<0\}$ do not intersect. Let $\alpha>0$ be such that,
for every $i\in\{1,\cdots,k\}$, $O\cap \{\phi_i<0\}\subset 
\{\alpha\ge\theta_i\ge \pi-\alpha\}$ with $\alpha>0$. 

Let $\phi$ and $\psi$ be negative and $\mu>0$. Since $\Omega$
has necks near each $p_i$ with $i\ge k+1$, there is in $\Omega\cap
\{\phi_i<\psi\}$ a geodesic $\Gamma_i$ of length less than $\mu$ 
joining the two adjacent arcs at $p_i$. Let $K$ be the compact part of
$\Omega$ delimited by the geodesic $\{\phi_i=\phi\}$ for $i\le k$ and the
geodesic $\Gamma_i$ for $i\ge k+1$. Besides we denote 
$$
O_{\phi,\psi}=O\backslash\left(\Big(\bigcup_{i=1}^k\{\phi_i<\phi\}
\Big)\bigcup
\Big(\bigcup_{i=k+1}^n \ \{\phi_i<\psi\}\Big)\right)
$$

From \eqref{inegalite}, we obtain:
\begin{align*}
\iint_{O_{\phi,\psi}}|X|^2\le \iint_{O(K,0)}|X|^2&\le
\int_{\partial_3(K,0)}u\langle X,\nu\rangle\\
&\le\sum_{i=1}^k\int_{O\cap\{\phi_i=\phi\}}u\langle X,\nu\rangle
+\sum_{i=k+1}^n\int_{O\cap\Gamma_i}u\langle X,\nu\rangle\\
&\le M\sum_{i=1}^k\int_{O\cap\{\phi_i=\phi\}} |X| +2M(n-k)\mu
\end{align*}

Thus letting $\mu$ going to $0$, $\psi$ going to $-\infty$ and denoting by
$O_\phi$ the subset $O_{\phi,-\infty}$ and $I_\phi=\bigcup_{i=1}^k
O\cap\{\phi_i=\phi\}$ a part of the boundary, we get 
\begin{equation}\label{equadiff}
\iint_{O_\phi}|X|^2 \le M\int_{I_\phi} |X|
\end{equation}
Let us denote by $\eta(\phi)$ the integral in the right-hand term. By
Schwartz's Lemma, we obtain:
$$
\eta^2(\phi)\le
\ell(I_\phi) \int_{I_\phi}|X|^2\le
C(\alpha)\int_{I_\phi}|X|^2
$$
where $C(\alpha)=k\int_{\alpha}^{\pi-\alpha}
\frac{d\theta}{\sin(\theta)}$. Thus $\int_{I_\phi}
|X|^2\ge \eta^2(\phi)/C(\alpha)$ and, in \eqref{equadiff},
this gives:
\begin{equation}
\mu_0+\int_\phi^0\frac{\eta^2(t)}{C(\alpha)}dt\le M\eta(\phi)
\end{equation}
with $\mu_0>0$. Let $\zeta$ be the function defined on
$I=(-(M^2C(\alpha))/ \mu_0,0]$ by :
$$
\frac{M}{\mu_0}-\frac{1}{\zeta(t)}=-\frac{t}{MC(\alpha)}
$$
This function $\zeta$ satisfies $\zeta(0)=\mu_0/M$ and
$\zeta'=-\zeta^2/(MC(\alpha))$. Thus for $\phi\in I$ we have
$\zeta(\phi)\le \eta(\phi)$. But $\eta(\phi)\le 2\ell(I_\phi)\le
2C(\alpha)$ and $\lim_{t\rightarrow -(M^2C(\alpha))/\mu_0}
\zeta(t)=+\infty$. We have a contradiction.
\end{proof}

We have a first lemma that allows us to bound admissible solutions.

\begin{lemma}\label{underbound}
Let $\Omega$ be an admissible domain in $\H^2$. Let $u$ be an admissible
solution with $B(u)=\emptyset$ and assume there exists $m\in\R$
such that $u\ge m$ on $\partial_\infty\Omega$. Then $u$ is
bounded below in $\Omega$.
\end{lemma}

\begin{proof}
There are only a finite number of points where such a lower-bound is
unknown: the vertices of $\Omega$ and the endpoints of arcs in $A(u)$. We
notice that there are only a finite number of such points. When an
endpoint of $A(u)$ or a vertex of $\Omega$ is in $\H^2$, a lower-bound is
given by the
 maximum principle for bounded domains. So let us consider an ideal vertex
$p$. Let $(\phi,\theta)$ be polar coordinates centered at $p$ and consider
$\Omega'=\Omega\cap \{\phi<0\}$. Let $m'\le m$ be such that $u\ge
m'$ on $\Omega\cap \{\phi=0\}$; let us prove that $u\ge m'$ in $\Omega'$.

Take $t<0$ and consider the minimal graph $w_t$ given by
Lemma~\ref{lem:barrier} on the domain $\{\phi>t\}$ which takes
the value $-\infty$ on $\{\phi=t\}$ and $m'$ on the
other boundary arc. We know that $w_t\le m'$ on $\{\phi>t\}$. By the
maximum principle for bounded domain, $w_t\le u$ on $\Omega'\cap
\{\phi>t\}$. As $t\rightarrow -\infty$, $w_t\rightarrow m'$;
hence $m'\le u$ on $\Omega'$.
\end{proof}

In the proof of Theorem~\ref{maxprinciple}, type 2 ideal vertices are the
hardest to deal with. Thus we need to be more precise for a bound near such
a vertex. In the following lemma, we use the minimal graph defined in
Lemma~\ref{lem:barrier} to control a minimal graph on one side of a type
2 ideal vertex.

\begin{lemma}\label{corridor}
For every $0<\bar{\theta}\le\pi/2$, there is a continuous increasing
function $H_{\bar{\theta}}:[0,\bar{\theta})\rightarrow \R_+$ with
$H_{\bar{\theta}}(0)=0$ such that the following is true.

Let $\Omega$ be an admissible domain in $\H^2$ and $p$ an ideal vertex
 of $\Omega$. We consider polar coordinates $(\phi,\theta)$ centered at
$p$. For $i=1,2$, let
$$
\gamma_i :
\begin{array}{cl}(0,1]&\longrightarrow\overline{\{\phi\le 0\}}^\infty\\
t&\longmapsto (\phi_i(t),\theta_i(t))
\end{array}
$$
be parametrizations of the two adjacent arcs in $\partial_\infty\Omega$
with $p$ as endpoint; we assume $\lim_{t\rightarrow 0}\gamma_i(t)=p$
$\gamma_i(1)\in\{\phi=0\}$ and $\theta_1(1)<\theta_2(1)$. Let
$\bar{\theta_2}=\liminf_{t\rightarrow 0}\theta_2(t)$; we
assume $\bar{\theta_2}>0$.

Let $u$ be an admissible solution on $\Omega$ such that $u\ge m$ in
$\gamma_1((0,1])$. Then for every $\theta_0$ and $\bar{\theta}$ with
$0<\theta_0<\bar{\theta}<\bar{\theta_2}$, there exists $\phi_0<0$ such that
:
$$
u\ge m-H_{\bar{\theta}}(\theta_0)\textrm{ on
}\Omega\cap\{\phi<\phi_0,\theta<\theta_0\}
$$
\end{lemma}

\begin{proof}
Let us consider $(\phi,\theta)$ polar coordinates
at a point in $\partial_\infty\H^2$ and $\bar{\theta}\in(0,\pi/2]$. On
$\Omega_{\bar{\theta}}=\{(\theta,\phi)\in\H^2| \theta<\bar{\theta}\}$,
we consider the minimal graph $h_{\bar{\theta}}(\phi,\theta)=
h_{\bar{\theta}}(\theta)$ given by Lemma~\ref{lem:barrier} with
$h_{\bar{\theta}}=0$ on $\{\theta=0\}$ and $\dfrac{\partial
h_{\bar{\theta}}}{\partial \nu}=+\infty$ along $\{\theta=\bar{\theta}\}$,
where $\nu$ is the outward pointing normal vector. For
$\theta_0<\bar{\theta}$, we
define: 
$$
H_{\bar{\theta}}(\theta_0)=h_{\bar{\theta}}(\theta_0+
\frac{\theta_0}{\bar{\theta}}(\bar{\theta}- \theta_0))=
\max_{\{0 \le\theta\le \theta_0+\frac{\theta_0}{\bar{\theta}}
(\bar{\theta}-\theta_0)\}} h_{\bar{\theta}}
$$
We remark that $\theta_0<\theta_0+\frac{\theta_0}{\bar{\theta}}
(\bar{\theta}-\theta_0)<\bar{\theta}$ when $0<\theta_0<\bar{\theta}$.
$H_{\bar{\theta}}$ is a continuous increasing function with
$H_{\bar{\theta}}(0)=0$.

\begin{figure}[h]
\begin{center}
\resizebox{0.5\linewidth}{!}{\input{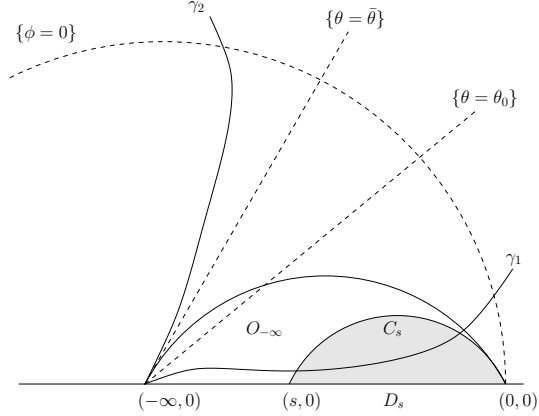}}
\caption{$O_s$ is the shadowed domain}\label{dessin5}
\end{center}
\end{figure}

Let $\Omega$, $u$, $(\phi,\theta)$ be as in the lemma. Let
$\bar{\theta}$ be less than $\bar{\theta_2}$; by changing $\phi$, we can
assume that $\theta_2(t)\ge\bar{\theta}$ for $t\in(0,1]$. Let $s$ be
negative, we consider the geodesic $B_s$ joining the points with polar
coordinates $(s,0)$ and $(0,0)$ and the arc $D_s$ in $\partial_\infty
\H^2\cap\{\phi\le 0\}$ joining both points.  Let $C_s$ be the equidistant
to $B_s$ which is at distance $d_{\bar{\theta}}$ (see \eqref{def:distance})
and is in the half-plane delimited by $B_s$ and $D_s$ (see
Figure~\ref{dessin5}). We denote by $O_s$ the domain bounded by
$C_s$ and $D_s$ ($O_s$ is included in $\theta\le \bar{\theta}$). On $O_s$,
we consider $k_s$ the minimal graph given by Lemma~\ref{lem:barrier} with
$k_s=0$ on $D_s$ and $\dfrac{\partial k_s}{\partial \nu}=+\infty$ on $C_s$.
We notice that $k_s>0$ on $O_s$. Since $\bar{\theta}<\theta_2(t)$ for every
$t$, the
boundary of $O_s\cap\Omega$ is composed of subarcs of $C_s$ and subarcs of
$\gamma_1$. Hence, by the maximum principle for bounded domains, $u\ge
m-k_s$ on $\Omega\cap O_s$. Let $s$ go to $-\infty$, $k_s$ converges to the
solution $k_{-\infty}$ on $O_{-\infty}$ with $h_{-\infty}=0$ on
$D_{-\infty}$ and $\dfrac{\partial h_{-\infty}}{\partial \nu}=+\infty$ on
$C_{-\infty}$ given by Lemma~\ref{lem:barrier}. Moreover, we have
$m-k_{-\infty}\le u$ on $\Omega\cap O_{-\infty}$. Fix
$0<\theta_0<\bar{\theta}$. Because of the definition of
$H_{\bar{\theta}}$, there is $\phi_0$ such that
$$
k_{-\infty}\le H_{\bar{\theta}}(\theta_0)\textrm{ on }\{\phi<\phi_0,
\theta<\theta_0\}
$$
which concludes the lemma.
\end{proof}

Actually, this Lemma says that if a solution is bounded below on one of
the two boundary components with $p$ as endpoint, then the solution is
bounded below in some ``sectorial'' neighborhood of this boundary
component.

Now we have the following result
\begin{proposition}\label{minoration}
Let $\Omega$ be an admissible domain and $u$ an admissible solution. Let
$p\in\partial\Omega$ be a type 2 ideal vertex of $\Omega$. We assume
there exists $m\in\R$ such that $u\ge m$ near $p$ on $\partial\Omega$.
Then, for every $\ve>0$, $u\ge m-\ve$ in a neighborhood of $p$
in $\Omega$.
\end{proposition}

\begin{proof}
Let $(\phi,\theta)$ be polar coordinates centered at $p$. We assume that
$u\ge m$
on $\partial\Omega\cap\{\phi\le 0\}$. Let $h$ be the minimal graph over
$\{\phi<0\}$ given by Lemma~\ref{lem:barrier} with boundary
values $h=-\infty$ on $\{\phi=0\}$ and $h=m$ on the other
boundary arc . For every $\ve>0$, we have $h\ge m-\ve$ on a
neighborhood of $p$, so it suffices to prove that $h\le u$ on $\Omega
\cap \{\phi<0\}$.

If $\{u<h\}$ is nonempty, consider $\ve>0$ a regular value of $h-u$
such that $\{u<h-\ve\}\neq\emptyset$. The only possible point-end of
$\{u<h-\ve\}$ is $p$. Let us prove that $p$ is in a corridor.
Let $\gamma_i=(\phi_i,\theta_i)$ be parametrizations defined on $(0,1]$ of 
both boundary arcs adjacent at $p$  in $\overline{\{\phi<0\}}^\infty$ with
$\lim_{t\rightarrow 0}\gamma_i(t)=p$, $\phi_1(1)=\phi_2(1)=0$ and
$\theta_1(1)<\theta_2(1)$. Since $p$ is of type 2, $\liminf_{t\rightarrow
0} \theta_2(t)>0$. Let $0<\bar{\theta}<\liminf_{t\rightarrow 0}
\theta_2(t)$, $H_{\bar{\theta}}$ be defined by
Lemma~\ref{corridor} and $\theta'\in(0,\bar{\theta})$ such that
$H_{\bar{\theta}}(\theta')<\ve$. Lemma~\ref{corridor} gives $\phi'<0$ such
that $u\ge m-H_{\bar{\theta}}(\theta')\ge m-\ve$ on $\Omega \cap
\{\phi<\phi',\theta<\theta'\}$. Applying Lemma~\ref{corridor} also on the
other side of $p$, we obtain $\phi_0<0$ and $\theta_0>0$ such
that $u\ge m-\ve$ in $\{\phi<\phi_0\}\cap \{\sin(\theta)<\sin(\theta_0)\}$.
Since $h\le m$ in $\{\phi<0\}$, we have $\{u<h-\ve\}\cap
(\{\phi<\phi_0\} \cap \{\sin(\theta)<\sin(\theta_0)\})=\emptyset$. Thus the
end is in a corridor. Theorem \ref{colkru} now implies that $u$ is not
bounded below near $p$, that contradicts Lemma~\ref{underbound} 
\end{proof}

We can now give the proof of the general maximum principle
(Theorem~\ref{maxprinciple}). We recall that the proof is written in the
case \eqref{case3}.

\begin{proof}[Proof of Theorem~\ref{maxprinciple}]
Let $\Omega$, $u_1$ and $u_2$ be as in the theorem and assume that
$u_2\le u_1$ is not true in the whole $\Omega$, so we can choose
$\ve>0$ such that $\{u_1\le u_2-\ve\}$ is nonempty. Since
$u_1>u_2-\ve$ on the arcs $D_i$, the point-ends of $\{u_1\le u_2-\ve\}$
are among the ideal vertices of $\Omega$. In particular, $\{u_1\le
u_2-\ve\}$ has a finite number of point-ends. Let us prove that each
point-end associated to a type 2 vertex of $\Omega$ is in a corridor.

Let $p$ be a point-end which is a type 2-i vertex of $\Omega$. Let
$\Gamma_1$ and $\Gamma_2$ denote the two components of
$\partial_\infty\Omega$ with $p$ as endpoint and consider polar coordinates
$(\phi,\theta)$ centered at $p$. There is $\phi_0$ such that
$$
u_1\ge \liminf_{\substack{x\in\Gamma_i\\ x\rightarrow p}} u_1-\ve/4
\ \textrm{ and }\  u_2\le \limsup_{\substack{x\in\Gamma_i\\ x\rightarrow
p}} u_2+\ve/4\ \textrm{ on }\Gamma_i\cap\{\phi<\phi_0\}
$$
Using Lemma~\ref{corridor} as in the proof of Lemma~\ref{minoration}, there
exist $\phi_1<\phi_0$ and $\theta_1\in(0,\pi/2)$ such that 

\begin{align*}
u_1&\ge \liminf_{\substack{x\in\Gamma_1\\ x\rightarrow p}}
u_1-\ve/2\textrm{ on } \Omega\cap\{\phi\le \phi_1,\theta<\theta_1\}\\
u_2&\le \limsup_{\substack{x\in\Gamma_1\\ x\rightarrow p}}
u_2+\ve/2\textrm{ on } \Omega\cap\{\phi\le \phi_1,\theta<\theta_1\}\\
u_1&\ge \liminf_{\substack{x\in\Gamma_2\\ x\rightarrow p}}
u_1-\ve/2\textrm{ on } \Omega\cap\{\phi\le \phi_1,\theta>\pi-\theta_1\}\\
u_2&\le \liminf_{\substack{x\in\Gamma_2\\ x\rightarrow p}}
u_2+\ve/2\textrm{ on } \Omega\cap\{\phi\le \phi_1,\theta>\pi-\theta_1\}
\end{align*}
Thus on $\Omega\cap\{\phi\le \phi_1,\theta<\theta_1\}$, we have
$$
u_1-u_2\ge \liminf_{\substack{x\in\Gamma_1\\ x\rightarrow p}} u_1-\ve/2-
(\limsup_{\substack{x\in\Gamma_1\\ x\rightarrow p}} u_2+\ve/2)\ge-\ve
$$
In $\Omega\cap\{\phi\le \phi_1,\theta>\pi-\theta_1\}$, we also have
$u_1-u_2>-\ve$. So $p$ is in a corridor.

In the case the point-end $p$ of $\{u_1\le u_2-\ve\}$ is a type 2-ii
vertex of $\Omega$, we can choose polar coordinates $(\phi,\theta)$
centered at $p$ such that the geodesic arc $A$ is in $\{\theta=\pi/2\}$ and
$\Gamma\subset
\overline{\{\theta<\pi/2\}}^\infty$. As above, we prove that there exist
$\phi_1$ and $\theta_1>0$ such that $u_1-u_2>-\ve$ in $\Omega\cap\{\phi\le
\phi_1,\theta<\theta_1\}$. So, $p$ is in a corridor.

Therefore, we have proved that either the point-ends of $\{u_1\le
u_2-\ve\}$ are in corridors or $\Omega$ has necks near them. Thus
Proposition~\ref{colkru} assures $u_1-u_2$ is not bounded below. 

Let $p$ be an ideal vertex of $\Omega$ of type 2-i. By
Lemma~\ref{underbound}, there are $m_1$
and $m_2$ in $\R$ such that $u_1\ge m_1$ and $u_2\le m_2$ in a neighborhood
of $p$ , so $u_1-u_2\ge m_1-m_2$ in a neighborhood of $p$. Since the
number of type 2-i vertices is finite, there is $m<0$ such that
$u_1-u_2\ge m$ in neighborhood of type 2-i vertices. Moreover $m$ can be
chosen to be a regular value for $u_1-u_2$. So let us denote the nonempty
set
$$
O=\{u_1-u_2\le m\}.
$$
In fact the value of $m$ is not already fixed : in the following, we shall
need to decrease $m$ a finite number of times (these changes are only
linked to the geometry of the domain).

We notice that $\partial O\cap(\cup_i D_i)=\emptyset$ and $\partial O\cap
(\cup_i C_i)\subset \overline{B(u_1)\cup A(u_2)}$. $O$ has a finite number
of point-ends which correspond to ideal vertices of type 1 or 2-ii. Let us
them denote by $p_1,\cdots,p_n$ and by $(\phi_i,\theta_i)$ polar
coordinates centered at $p_i$. As in the proof of Proposition~\ref{colkru},
for $\delta>0$ small, we denote by $N_\delta$ the closed
$\delta$-neighborhood of $\overline{B(u_1)\cup A(u_2)}$ and we define:
$$
O(\phi,\delta)=O\backslash \big(N_\delta\bigcup (\cup_i\{\phi_i\le
\phi\})\big) 
$$
Its boundary $\partial O(\phi,\delta)\subset \Omega$ is piecewise smooth
and is composed of three parts:
\begin{itemize}
\item $\partial_1(\phi,\delta)=\partial O(\phi,\delta)\cap
\partial O$, where $u_2-u_1=-m$,
\item $\partial_2(\phi,\delta)=\partial O(\phi,\delta)\cap
\partial N_\delta$,
\item $\partial_3(\phi,\delta)=\partial O(\phi,\delta)\cap (
\cup_i\{\phi_i= \phi\} \backslash \partial O)$.
\end{itemize}
We call $X=X_{u_2}-X_{u_1}$ and $\nu$ the outgoing normal to
$\partial O(\phi,\delta)$. We have:
$$
0=\int_{\partial O(\phi,\delta)} \langle X,\nu\rangle=\int_{\partial_1
(\phi,\delta)} \langle X,\nu\rangle + \int_{\partial_2 (\phi,\delta)}
\langle X,\nu\rangle+ \int_{\partial_3 (\phi,\delta)} \langle X,\nu\rangle 
$$
We notice that along $\partial_1(\phi,\delta)$, $\nabla u_2-\nabla u_1$
points into $O$ so $X$ points to $O$. Hence $\langle X,\nu\rangle$ is
negative on $\partial_1(\phi,\delta)$ (see Lemma 2 in \cite{ck1}).
Besides, we have $|X|\le 2$ and the length of $\partial_3(\phi,\delta)$
is uniformly bounded for fixed $\phi$ since either the point-ends of $O$
are in corridors or $\Omega$ has necks at them.
Thus, with $K=\cap_i\{\phi_i>\phi\}$, Claim~\ref{limint} implies that,
letting $\delta$ goes to $0$, we obtain:
$$
0=\int_{\partial_1 (\phi,0)} \langle X,\nu\rangle+ \int_{\partial_3
(\phi,0)} \langle X,\nu\rangle
$$
Or
$$
0<-\int_{\partial_1 (\phi,0)} \langle X,\nu\rangle = \int_{\partial_3
(\phi,0)} \langle X,\nu\rangle
$$
We can decomposed $\partial_3 (\phi,0)$ in a finite number of parts 
$\gamma_1(\phi),\cdots,\gamma_n(\phi)$: $\gamma_i(\phi)$ is the part of
$\partial_3 (\phi,0)$ in $\{\phi_i=\phi\}$. Thus we have:
$$
-\int_{\partial_1 (\phi,0)} \langle X,\nu \rangle
=\sum_{i=1}^n\int_{\gamma_i(\phi)} \langle X,\nu\rangle
$$
The left-hand term is positive and increases as $\phi\searrow -\infty$.
Thus we get a contradiction and Theorem~\ref{maxprinciple} is proved once
we have established the following claim:
\begin{claim}\label{claimfinal}
For every $i$, we have
$$
 \limsup_{\phi\rightarrow -\infty} \int_{\gamma_i(\phi)} \langle
X,\nu\rangle \le 0 
$$
\end{claim}

First we suppose $p_i$ is a type 1 vertex. Let $\phi_0<0$ be fixed. Since
$p_i$ is a type 1 vertex, for each $\mu>0$ there is a geodesic arc $\Gamma
\subset \Omega\cap \{\phi<\phi_0\}$ of length less than
$\mu$. $\Gamma$ separates $\Omega\cap \{\phi<\phi_0\}$ into a non compact
component and a compact part $\Omega_\Gamma$. Let $\phi_1<\phi_0$ be
such that $\Gamma\in\{\phi>\phi_1\}$. As above we can compute the
flux of $X$ along the boundary of $O\cap \Omega_\Gamma$ and we get:
$$
0=\int_{\partial(O\cap\Omega_\Gamma)} \langle X, \nu'\rangle
=\int_{\partial_1(\phi_1, 0)\cap\Omega_\Gamma}\langle X,\nu'\rangle
+\int_{O\cap\Gamma} \langle X,\nu'\rangle -\int_{\gamma_i(\phi_0)}
\langle X,\nu\rangle
$$
with $\nu'$ the outgoing normal from $O\cap \Omega_\Gamma$. The sign of
the last term comes from the fact that $\nu'=-\nu$ along $\gamma_i(\phi)$.
As above, $X$ points to $O\cap \Omega_\Gamma$ along $\partial_1(\phi_1,
0)\cap\Omega_\Gamma$, thus $\int_{\partial_1(\phi_1, 0)\cap\Omega_\Gamma}
\langle X,\nu'\rangle \le 0$ and
$$
\int_{\gamma_i(\phi_0)}\langle X,\nu\rangle =\int_{\partial_1(\phi_1,
0)\cap\Omega_\Gamma} \langle X,\nu'\rangle + \int_{O\cap\Gamma}
\langle X,\nu'\rangle \le 2\ell(\Gamma)\le 2\mu
$$
The above inequality occurs for every $\mu>0$. Then
$\int_{\gamma_i(\phi_0)} \langle X,\nu\rangle\le 0$ and the claim is proved
when $p_i$ is a type $1$ vertex of $\Omega$.

Let us now suppose $p_i$ is a type 2-ii vertex of $\Omega$. We choose the
polar coordinates centered $p_i$ such that the geodesic arc
$A$ is in $\{\theta=\pi/2\}$ and the arc $\Gamma$ is in
$\{\theta<\pi/2\}$. We fix $\phi_0<0$.
Let $G:[0,1]\rightarrow\H^2\cup\partial_\infty\H^2$ be a parametrization of
$\Gamma$, in polar coordinates $G(t)=(\phi(t),\theta(t))$ for $t>0$ with
$\phi(1)=\phi_0$. Since $p_i$ is an endpoint of $\Gamma$,
$\lim_{t\rightarrow 0}\phi(t)=-\infty$. Let $\theta_\infty$ be
$\limsup_{t\rightarrow
0}\theta(t)$. If $\theta_\infty=\pi/2$, we have $\liminf_{t\rightarrow
0}d(G(t),A)=0$ as in type 1 vertices and we can apply the above proof.

We then assume $\theta_\infty<\pi/2$. Let us consider
$\bar{\theta}\in(\theta_\infty, \pi/2)$. By changing $\phi_0$, we
can assume that $\theta(t)<\bar{\theta}$ for every $t\in(0,1]$. 

Let us define $u_1^\infty=\liminf_{\substack{x\in\Gamma\\ x\rightarrow
p}} u_1(x)$ and $u_2^\infty=\limsup_{\substack{x\in\Gamma\\ x\rightarrow
p}} u_2(x)$. From Lemma~\ref{corridor} and Proposition~\ref{minoration},
there
are $\bar{\phi}<\phi_0$  and $\overline{m}\ge 1$ such that $u_1\ge
u_1^\infty-1$ on $\Omega\cap\{\phi<\bar{\phi}\}$
and $u_2\le u_2^\infty+\overline{m}$ on $\Omega\cap \{\phi<\bar{\phi},
\theta<\bar{\theta}\}$. Thus on $\Omega\cap\{\phi<\bar{\phi},
\theta\le\bar{\theta}\}$, $u_1-u_2\ge u_1^\infty-1-u_2^\infty-\bar{m}\ge
-1-\overline{m}$. So, if $m$ is chosen less than $-1-\overline{m}$, we
have $(O\cap\{\phi\le\bar{\phi}\})\subset \{\bar{\theta}\le\theta\le
\pi/2\}$.

We can change the polar coordinate $\phi$ to have $\bar{\phi}=0$. Let
$\Omega_1$ be the domain bounded by the geodesic joining $p_i$ to
the point $p_-$ of polar coordinates $(a,\pi)$ ($a<0$) and the equidistant
to this geodesic which is at distance $d_{\theta_\infty}$ (see
\eqref{def:distance}) such that $\Omega\cap\Omega_1\neq\emptyset$. Here,
$a$ is chosen such that
$\Omega_1\subset\{\phi<0\}$ (see Figure~\ref{dessin3}).
By Lemma~\ref{lem:barrier}, there exists the minimal graph $h_1$ define
d on $\Omega_1$ with value $+\infty$ on the geodesic boundary component and
value $u_1^\infty-1$ on the equidistant. Let $\Omega_2$ be the domain
delimited
by the geodesic joining $p_i$ to the point $p_+$ of polar coordinates
$(a,0)$ and the arc in $\partial_\infty\H^2$ joining $p_i$ to $p_+$ (
\textit{i.e.}
in polar coordinates, $(-\infty,a)\times\{0\}$). On $\Omega_2$, we consider
the minimal graph $h_2$ with value $+\infty$ on the geodesic boundary
component and
$u_2^{\infty}+1$ on the arc in $\partial_\infty\H^2$. As in the proof of
Lemma~\ref{corridor}, we ca deduce $h_1\le u_1$ in $\Omega\cap\Omega_1$
and
$u_2\le h_2$ on $\Omega\cap\Omega_2$. Hence $u_1-u_2\ge h_1-h_2$ in
$\Omega\cap \Omega_1\cap \Omega_2$ so let
us bound $h_1-h_2$ below in $\Omega_1\cap \Omega_2$.

First, because of the definition of $\Omega_1$, there is $\bar{\phi_0}$
such that $O\cap\{\phi\le\bar{\phi_0}\}\subset\{\phi\le\bar{\phi_0},
\bar{\theta}\le\theta\le\pi/2\}\subset \Omega_1$.

\begin{figure}[h]
\begin{center}
\resizebox{0.7\linewidth}{!}{\input{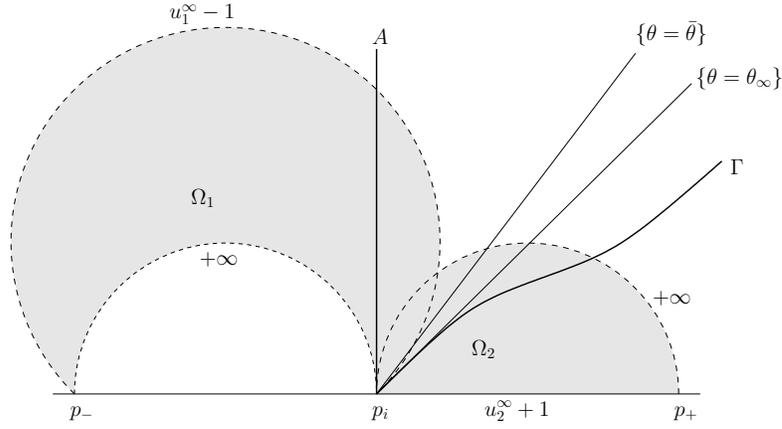}}
\caption{The domains $\Omega_1$ and $\Omega_2$ in $\H^2$} \label{dessin3}
\end{center}
\end{figure}

To make some computations, we use other coordinates : we consider
$\H^2=\R\times\R_+^*$ with the classical hyperbolic metric such that $p$ is
the infinity, $p_+=(1,0)$ and $p_-=(-1,0)$. We have
$\Omega\subset\R_+^*\times\R_+^*$ near $p$, $\Omega_1=\{(x,y)\in
(-1,+\infty)\times\R_+^*|y>\tan(\theta_\infty) (x+1)\}$ and
$\Omega_2=(1,+\infty)\times \R_+^*$. In fact, the points of polar
coordinates $(\phi,\theta)$ becomes
$(x,y)=e^{-(\phi-a)}(\cos(\theta),\sin(\theta))$. The functions
$h_1$ and $h_2$ have the following expressions (see \eqref{expressionh}): 
\begin{align*}
h_1(x,y)&=\ln\left(\sqrt{1+\left(\frac{y}{x+1}\right)^2}+\frac{y}{x+1}
\right)- c_{\theta_\infty} +u_1^\infty-1 \\
h_2(x,y)&=\ln\left(\sqrt{1+\left(\frac{y}{x-1}\right)^2}+\frac{y}{x-1}
\right)+u_2^\infty+1
\end{align*}
where $c_{\theta_\infty}$ is a constant which depends only on
$\theta_\infty$.

With $a_1=y/(x+1)$ and $a_2=y/(x-1)$ this gives:
\begin{align*}
h_1(x,y)-h_2(x,y)&=\ln\left(\frac{\sqrt{1+a_1^2}+a_1}{\sqrt{1+a_2^2}
+a_2}\right)-c_{\theta_\infty} +u_1^\infty-1-u_2^\infty-1\\
&\ge \ln\left(\frac{\sqrt{1+a_1^2}+a_1}{\sqrt{1+a_2^2}
+a_2}\right)-c_{\theta_\infty} -2
\end{align*}

We have $a_2/a_1=(x+1)/(x-1)$ thus on $\{x\ge 2\}$, $1\le a_2/a_1\le 3$.
So, on $\{x\ge 2\}$:
$$
\frac{1}{3}\le \frac{\sqrt{1+a_1^2}+a_1}{\sqrt{1+a_2^2}
+a_2} \le 1
$$
and $h_1(x,y)-h_2(x,y)\ge -\ln 3-c_{\theta_\infty} -2$ on $\{x\ge
2\}\cap(\Omega_1\cap\Omega_2)$. Thus if $m$ is chosen to be less than $-\ln
3-c_{\theta_\infty} -2$, we have:
$$
(O\cap\{\phi\le\phi_0\})\subset \{0\le x\le 2\}
$$
Then $\lim_{\phi\rightarrow-\infty} \ell(\gamma_i(\phi))=0$. This gives 
Claim~\ref{claimfinal} since :
$$
\left|\int_{\gamma_i(\phi))} \langle X,\nu\rangle \right|\le
2\ell(\gamma_i(\phi))\xrightarrow[\phi\rightarrow -\infty]{}0
$$
This completes the proof of Theorem~\ref{maxprinciple}.
\end{proof}

This maximum principle gives immediately a lower-bound result and
a uniqueness result:
\begin{corollary}\label{minoration2}
Let $\Omega$ be an admissible domain and $u$ an admissible solution.
We assume there exists $m\in\R$ such that $u\ge m$ on
$\partial_\infty\Omega$. Then $u\ge m$ in $\Omega$.
\end{corollary}

\begin{corollary}
Let $\Omega\subset\H^2$ be an admissible domain and $u_1$ and $u_2$ be two
admissible solutions. We assume that $u_1= u_2$ on $\partial_\infty\Omega$.
Besides we assume that the behaviour near each ideal vertex
$p\in\partial_\infty\H^2$ is one of the following.
\begin{enumerate}
\item[\textbf{type 1}] $\Omega$ has necks near $p$;
\item[\textbf{type 2-i}] we have $\lim_p u_1=\lim_p u_2$ exists and is
finite along both boundary components with $p$ as endpoint;
\item[\textbf{type 2-ii}] if $A\subset A(u_1)(= A(u_2))$ (resp. $B\subset
B(u_1)(=
B(u_2))$) is a geodesic arc with $p$ as endpoint and $\Gamma$ is the
other boundary arc with endpoint $p$ that bounds $\Omega$ near $p$, we
have $\lim_p u_1=\lim_p u_2$  exists and is finite along $\Gamma$ and .
\end{enumerate}
Then we have $u_1= u_2$ in $\Omega$.
\end{corollary}

%%%%%%%%%%%%%%
\subsubsection{A counterexample}\label{seccontreexemple}
%%%%%%%%%%%%%%

In this section, we construct a counterexample to a general maximum
principle. To be more precise we have the following result:
\begin{proposition}\label{contreexemple}
There is a continuous function on $\partial_\infty\H^2$ minus two points
that admits several minimal extensions to $\H^2$.
\end{proposition}
We remark that any such function admits a minimal extension to $\H^2$ by
Theorem~\ref{th3}. The idea to construct several extensions comes from
Collin's construction in \cite{col2}.

In the following, we shall work in the disk model for $\H^2$. Let us fix
$\alpha$ in $(\pi/4,\pi/2)$, we denote $z_\alpha=e^{i\alpha}$ the points
in $\partial_\infty\H^2$. Let us
consider the ideal rectangle $R_\alpha$ with the
points $z_\alpha,-\overline{z_\alpha},-z_\alpha$ and $\overline{z_\alpha}$
as vertices. This domain is symmetric with respect to the geodesics
$\{\re z=0\}$ and $\{\im z=0\}$. We can extend the domain $R_\alpha$
by
reflection along the "vertical" geodesics $(z_\alpha,
\overline{z_\alpha})$ and $(-\overline{z_\alpha},-z_\alpha)$ and their
images by these reflections. We obtain a domain $\Delta_\alpha$ which is
invariant under the translation $t$ along the geodesic $\{\im z=0\}$
defined
by $t(-\overline{z_\alpha})=z_\alpha$. We then denote by $p_0$ the
point $-z_\alpha$ and by $q_0$ the point $-\overline{z_\alpha}$; for
$n\in \Z$, we define $p_n$ and $q_n$ by $p_n=t^n(p_0)$ and $q_n=t^n(q_0)$
(see Figure~\ref{dessin4}).

We have a first lemma.
\begin{lemma}\label{defwlambda}
There exists a family of minimal graph $w_\lambda$ over
$\Delta_\alpha$ such that
\begin{itemize}
\item $w_\lambda$ takes on the geodesics $(p_k,p_{k+1})$ and
$(q_k,q_{k+1})$ the value $+\infty$ if $k$ is even and $-\infty$ is $k$ is
odd,
\item $w_\lambda=k\lambda$ on the geodesic $(p_k,q_k)$,
\item the graph of $w_\lambda$ is invariant by the translation of
$\H^2\times\R$ defined by $(p,z)\mapsto (t^2(p),z+2\lambda)$.
\end{itemize}
\end{lemma}

\begin{proof}
Since $\alpha\in(\pi/4,\pi/2)$, the rectangle $R_\alpha$ satisfies the
hypotheses of Theorem~\ref{th2}. So, for every $\lambda\in\R$, we can
construct a minimal graph $w_\lambda$ on $R_\alpha$ with boundary data
$+\infty$ on $(p_0,p_1)$ and $(q_0,q_1)$, $0$ on $(p_0,q_0)$ and $\lambda$
on $(p_1,q_1)$. Since $w_\lambda$ is constant on $(p_0,q_0)$ and
$(p_1,q_1)$, we can extend the definition of $w_\lambda$ to
$\Delta_\alpha$ by Schwartz reflection. The properties of $w_\lambda$ are
deduced easily from its contruction.
\end{proof}

Let $H$ be a horocycle at a vertex $p_n$ of $\Delta_\alpha$, we then
define $p_n^-=H\cap (p_{n-1},p_n)$ and $p_n^+=H\cap (p_n,p_{n+1})$; in the
same way we define $q_n^-$ and $q_n^+$.

Let $D_\alpha$ be the domain bounded by the geodesics $(p_0,q_0)$ and
$(p_1,q_1)$ and the arcs in $\partial_\infty\H^2$ joining $p_0$ to $p_1$
and $q_0$ to $q_1$. We have a second lemma.
\begin{lemma}\label{defmk}
Let us consider at each vertex of $R_\alpha$, $p_0,p_1,q_0$ and $q_1$,
a horocycle (they are assumed to be disjoint). Let us fix $\ve>0$.
Then there exist $m>0$ and $\beta\in(\alpha,\pi/2)$ such that the
following is true. Let $u$ be a minimal graph over $D_\alpha$ which is
continuous up to $\partial_\infty D_\alpha$ minus the four vertices with:
\begin{itemize}
\item $u=m$ on the boundary subarcs of $\partial_\infty\H^2$ joining
$e^{i\beta}$ to $-e^{-i\beta}$
and $-e^{i\beta}$ to $e^{-i\beta}$,
\item $u\le m$ on $\partial_\infty D_\alpha$,
\item $u\le 0$ on $(p_0,q_0)$ and $(p_1,q_1)$.
\end{itemize}
Then:
\begin{align*}
\int_{[p_0^+,p_1^-]}\langle X_u, \nu\rangle &\ge \ell([p_0^+,p_1^-])-\ve &
\int_{[q_0^+,q_1^-]}\langle X_u, \nu\rangle &\ge \ell([q_0^+,q_1^-])-\ve
\end{align*}
with $\nu$ the outgoing normal from $R_\alpha$ and $[p_0^+,p_1^-]$ denotes
the segment in the geodesic $(p_0,p_1)$ joining $p_0^+$ to $p_1^-$.
\end{lemma}

\begin{proof}
If the lemma is false, for every $n\in\N$, there is a minimal graph $u_n$
on $D_\alpha$ continuous up to $\partial_\infty D_\alpha$ minus
the four vertices with:
\begin{itemize}
\item $u_n=n$ on the boundary arcs joining $e^{i\beta_n}$ to
$-e^{-i\beta_n}$ and $-e^{i\beta_n}$ to $e^{-i\beta_n}$ where
$\beta_n=\alpha+1/n$,
\item $u\le n$ on $\partial_\infty D_\alpha$,
\item $u\le 0$ on $(p_0,q_0)$ and $(p_1,q_1)$,
\item $\displaystyle\int_{[p_0^+,p_1^-]}\langle X_{u_n}, \nu \rangle\le
\ell([p_0^+,p_1^-])-\ve$ or $\displaystyle
\int_{[q_0^+,q_1^-]}\langle X_u, \nu\rangle \le \ell([q_0^+,q_1^-])-\ve$.
\end{itemize} 
We recall that $w_0$ is defined over $R_\alpha$ with $w_0=0$ on
$(p_0,q_0)$ and $(p_1,q_1)$ and $w_0=+\infty$ on $(p_0,p_1)$ and
$(q_0,q_1)$. Thus by the maximum principle (Theorem~\ref{maxprinciple}),
for every $n\in\N$, $u_n\le w_0$:
the sequence $u_n$ is bounded above on $R_\alpha$. Let $h_n$ be the
minimal graph over the domain in $D_\alpha\backslash R_\alpha$ bounded by
the geodesic $(-e^{i\beta_n},e^{-i\beta_n})$ and the arc in
$\partial_\infty\H^2$ joining $-e^{i\beta_n}$ to $e^{-i\beta_n}$ with
boundary value $-\infty$ on the geodesic and $n$ on the subarc of
$\partial_\infty\H^2$. By the maximum principle, for every $n\in\N$,
$u_n\ge
h_n$. Since $\beta_n\rightarrow\alpha$, $u_n\rightarrow+\infty$ on the
domain bounded by the geodesic $(p_0,p_1)$ and the arc in
$\partial_\infty\H^2$ joining $p_0$ to $p_1$. This implies that:
$$
\int_{[p_0^+,p_1^-]}\langle X_{u_n}, \nu\rangle \longrightarrow
\ell([p_0^+,p_1^-])
$$
In the same way we prove that:
$$
\int_{[q_0^+,q_1^-]}\langle X_{u_n}, \nu\rangle \longrightarrow
\ell([q_0^+,q_1^-])
$$
This a contradiction and the lemma is proved.
\end{proof}

We can now prove Proposition~\ref{contreexemple}.

\begin{proof}
For every $n\in\N$, we denote by $\Omega_n$ the domain bounded by the
geodesic $(p_0,q_0)$ and $(p_n,q_n)$ and the arcs in $\partial_\infty\H^2$
joining $p_0$ to $p_n$ and $q_0$ to $q_n$, finally we define
$\Omega_\infty=\cup_n\Omega_n$ ($\Omega_\infty$ is a half-plane). Let $o$
be the endpoint of the geodesic $\{y=0\}$ in the ideal boundary
of $\Omega_\infty$. In the following we define a continuous function $f$ on
$\partial_\infty\Omega_\infty\backslash\{o\}$ which admits two
minimal extensions in $\Omega_\infty$; we shall have $f=0$ on $(p_0,q_0)$
thus, by Schwartz reflection, the definition will extend to $\H^2$ and the
proposition will be proved.

For every $n\in \N$, we choose $H(p_n)$ a horocycle centered at $p_n$. By
symmetry with respect to the geodesic $\{y=0\}$ we define $H(q_n)$ a
horocycle centered at $q_n$. Let $p_n^0$ and $q_n^0$ be the intersections
of
the geodesic $(p_n,q_n)$ with $H(p_n)$ and $H(q_n)$. We also define
$h(p_n)$(resp. $h(q_n)$) as the arc of $H(p_n)$ (resp. $H(q_n)$) between
$p_n^-$ and $p_n^+$ (resp. $q_n^-$ and $q_n^+)$ (see Figure~\ref{dessin4}).

\begin{figure}[h]
\begin{center}
\resizebox{0.7\linewidth}{!}{\input{dessin4.pstex_t}}
\caption{}\label{dessin4}
\end{center}
\end{figure}

Let us consider $w=w_1$ and $w'=w_{-1}$ where $w_{\pm1}$ are defined by
Lemma~\ref{defwlambda}. On $\Omega_\infty\cap D_\alpha$, $w\ge w'$ and
$w=0=w'$ on $(p_0,q_0)$, thus $X_{w'}-X_{w}$ points out of $\Omega_\infty$.
This implies that we can choose suitable $H(p_k)$ and
a positive sequence $(\ve_k)_{k\in\N}$ such that:
$$
0<\sum_{k\ge 0} \ve_k +\sum_{k\ge 0} \ell(h(p_k))+\sum_{k\ge 0}
\ell(h(q_k))< \frac{1}{5}\int_{[p_0^0,q_0^0]} \langle(X_{w'}-X_w),
\nu\rangle=\ve
$$
with $\nu$ the out-going normal from $\Omega_\infty$.

For every $k$, Lemma~\ref{defmk} associates to $\ve_k$ and $H(p_k),
H(p_{k+1}), H(q_k)$ and $H(q_{k+1})$ two real numbers $m_k>0$ and
$\beta_k\in(\alpha,\pi/2)$. Let $I_k$ be the image by $t^k$ of the arcs in
$\partial_\infty D_\alpha$ joining $e^{i\beta_k}$ to $-e^{-i\beta_k}$
and $-e^{i\beta_k}$ to $e^{-i\beta_k}$ and $J_k$ the image by $t^k$ of the
others arcs in $\partial_\infty D_\alpha\cap\partial_\infty\H^2$.

Let us define on $\partial_\infty\Omega_\infty\backslash\{o\}$ a continuous
function $f$ which satisfies
\begin{itemize}
\item $f=(-1)^k(m_k+(k+1))$ on $I_k$,
\item $|f|\le m_k+(k+1)$ on $J_k$,
\item $f=0$ on $(p_0,q_0)$.
\end{itemize}

For every $n\in\N$, we define on $\Omega_n$ the minimal graph $u_n$ and
$u_n'$ with boundary value $u_n=u_n'=f$ on $\partial_\infty \Omega_\infty
\cap \partial_\infty \Omega_n$ and $u_n=+\infty$ and $u_n'=-\infty$ on
$(p_n,q_n)$, these minimal graphs exist because of Theorem \ref{th2}. By
the maximum principle (Theorem~\ref{maxprinciple}), we have
$u_n\ge u_n'$ and
$\{u_n\}$ (resp. $\{u_n'\}$) is a decreasing sequence (resp. increasing
sequence). Hence they converge to minimal graphs $u$ and $u'$ on
$\Omega_\infty$ with $f$ as boundary value. Let us prove that $u\neq u'$.

To do this, let us introduce some comparison functions; first we need some
new domains : for every $n>0$ we define 
\begin{align*}
B_n&=\left(\bigcup_{0\le 2k+1\le
n}t^{2k+1}(\overline{R_\alpha})\right)\cup\left(\bigcup_{0\le 2k\le
n}t^{2k}(\overline{D_\alpha})\right)\\
B_n'&=\left(\bigcup_{0\le 2k\le
n}t^{2k}(\overline{R_\alpha})\right)\cup\left(\bigcup_{0\le 2k+1\le
n}t^{2k+1}(\overline{D_\alpha})\right)
\end{align*}
On $B_n$, we define the minimal graph $v_n$ with boundary values $-\infty$
on $(p_k,p_{k+1})\cup(q_k,q_{k+1})$ if $k\le n$ and $k$ odd, $n+1$ on
$(p_{n+1},q_{n+1})$ and $f$ on the remainder of $\partial_\infty B_n$. On
$B_n'$, we define the minimal graph $v_n'$ with boundary value $+\infty$ on
$(p_k,p_{k+1})\cup(q_k,q_{k+1})$ if $k\le n$ and $k$ even, $-(n+1)$ on
$(p_{n+1},q_{n+1})$ and $f$ on the remainder of $\partial_\infty B_n'$. We
notice that these minimal graphs exist : Theorem~\ref{th2} can be applied
because of the existence of $w$.

On $\partial \Delta_\alpha\cap\overline{B_n}$, we have $v_n\le w$. Thus by
Theorem~\ref{maxprinciple}, $v_n\le w$ in $\Delta_\alpha\cap B_n$. Hence,
for every $0\le k\le n$, $v_n \le k$ on
$(p_k,q_k)$. Let us fix $k$ an even integer less than $n$; we have $v_n\le
k+1$  on $(p_k,q_k)\cup(p_{k+1},q_{k+1})$ and $v_n=f=m_k+(k+1)$ on $I_k$,
thus by Lemma~\ref{defmk} applied to $t^k(D_\alpha)$ we obtain:
\begin{align}
\int_{[p_k^+,p_{k+1}^-]}\langle X_{v_n},\nu\rangle & \ge
\ell([p_k^+,p_{k+1}^-])-\ve_k \label{eq1}\\
 \int_{[q_k^+,q_{k+1}^-]}\langle X_{v_n}, \nu\rangle & \ge
\ell([q_k^+,q_{k+1}^-])-\ve_k \label{eq2}
\end{align}
With $\nu$ the outgoing normal from $\Delta_\alpha$. When $k$ is odd,
we have 
\begin{align}\label{eq3}
\int_{[p_{k}^+,p_{k+1}^-]}\langle X_{v_n}, \nu\rangle & =
-\ell([p_{k}^+,p_{k+1}^-]) &
\int_{[q_k^+,q_{k+1}^-]}\langle X_{v_n}, \nu\rangle & =
-\ell([q_k^+,q_{k+1}^-])
\end{align}
Let $\Gamma_n$ be the closed curve in $\overline{B_n}$ composed of the
geodesic arcs $[p_0^0,q_0^0]$, $[p_k^+,p_{k+1}^-]$ for $0\le k\le n$,
$[p_{n+1}^0,q_{n+1}^0]$ and $[q_k^+,q_{k+1}^-]$ for $0\le k\le n$ and the
arcs of horocycles $h(p_k)\cap B_n$ and $h(q_k)\cap B_n$ for $0\le k\le
n+1$. By Stokes theorem $\int_{\Gamma_n}\langle (X_{v_n}-X_w),\nu
\rangle=0$ with $\nu$ the outgoing normal, so we have :
\begin{align*}
0&=\int_{\Gamma_n}\langle (X_{v_n}-X_w),\nu\rangle\\
&=\int_{[p_0^0,q_0^0]}\langle(X_{v_n}-X_w),\nu\rangle+
\int_{[p_{n+1}^0,q_{n+1}^0]} \langle(X_{v_n}-X_w),\nu\rangle \\
&\quad\quad+\sum_{k=0}^n\left(\int_{[p_k^+,p_{k+1}^-]}\langle (X_{v_n}
-X_w),\nu\rangle + \int_{[q_k^+,q_{k+1}^-]}\langle(X_{v_n}
-X_w),\nu\rangle\right)\\
&\quad\quad +\sum_{k=0}^{n+1}\left(\int_{h(p_k)\cap B_n}\langle (X_{v_n}
-X_w),\nu\rangle + \int_{h(q_k)\cap B_n}\langle (X_{v_n}
-X_w),\nu\rangle\right)\\
&\textrm{ since }X_{v_n}-X_w\textrm{ points out of }B_n\textrm{ along
}(p_{n+1},q_{n+1})\\
&\ge \int_{[p_0^0,q_0^0]}\langle(X_{v_n}-X_w),\nu\rangle\\
&\quad\quad+\sum_{\substack{k=0\\k
\textrm{ even}}}^n\left(\int_{[p_k^+,p_{k+1}^-]}\langle(X_{ v_n }
-X_w),\nu\rangle+ \int_{[q_k^+,q_{k+1}^-]}\langle(X_{v_n}
-X_w),\nu\rangle\right)\\
&\quad\quad+\sum_{\substack{k=0\\k
\textrm{ odd}}}^n\left(\int_{[p_k^+,p_{k+1}^-]}\langle(X_{ v_n }
-X_w),\nu\rangle+ \int_{[q_k^+,q_{k+1}^-]}\langle(X_{v_n}
-X_w),\nu\rangle\right)\\
&\quad\quad-\sum_{k=0}^{n+1}\left(2\ell(h(p_k))
+2\ell(h(q_k))\right)\\ 
&\textrm{because of \eqref{eq1},\eqref{eq2} and \eqref{eq3}}\\
&\ge \int_{[p_0^0,q_0^0]}\langle(X_{v_n}-X_w),\nu\rangle -
\sum_{\substack{k=0\\k
\textrm{ even}}}^n 2\ve_k -2\sum_{k=0}^{n+1}\left(\ell(h(p_k))
+\ell(h(q_k))\right)
\end{align*}
Thus since $X_{v_n}-X_w$ points out of $\Omega_n$ along $(p_0,q_0)$:
$$
0\le \int_{[p_0^0,q_0^0]}\langle (X_{v_n}-X_w),\nu\rangle \le 2\left(
\sum_{\substack{k=0\\k \textrm{ even}}}^n \ve_k
+\sum_{k=0}^{n+1}\ell(h(p_k))
+\ell(h(q_k))\right)\le 2\ve
$$

Now, on $\partial B_n$ we have $u_n\ge v_n$. So, by
Theorem~\ref{maxprinciple}, $u_n\ge v_n$ on $B_n$. This implies that
$X_{v_n}-X_{u_n}$ points out $B_n$ along $(p_0,q_0)$ and 
$$
\int_{[q_0^0,p_0^0]}\langle X_{u_n},\nu\rangle\le
\int_{[q_0^0,p_0^0]}\langle X_{v_n},\nu\rangle\le
\left(\int_{[q_0^0,p_0^0]}\langle X_{w},\nu\rangle\right)+2\ve
$$
Thus for the limit $u$, we have:
$$
\int_{[q_0^0,p_0^0]}\langle X_{u},\nu\rangle\le
\left(\int_{[q_0^0,p_0^0]}\langle X_{w},\nu\rangle \right)+2\ve
$$

Working with $u_n'$, $v_n'$ and $w'$ on $B_n'$ in the same way we prove
that :
$$
\int_{[q_0^0,p_0^0]}\langle X_{u'},\nu\rangle\ge
\left(\int_{[q_0^0,p_0^0]}\langle X_{w'},\nu\rangle \right)-2\ve
$$
Thus:
$$
\int_{[q_0^0,p_0^0]}\langle(X_{u'}-X_{u}),\nu\rangle \ge
\left(\int_{[q_0^0,p_0^0]}\langle (X_{w'}-X_{w}),\nu \rangle\right)-4\ve>0
$$
This implies that $X_u\neq X_{u'}$ on $[q_0^0,p_0^0]$ and $u\neq u'$ on
$\Omega_\infty$
\end{proof}

\appendix

%%%%%%%%%%%%%%%%%%%
%%%%%%%%%%%%%%%%%%%

\section{CMC graphs in $\H^2\times \R$ invariant under
translations} \label{appendix}

%%%%%%%%%%%%%%%%%%%
%%%%%%%%%%%%%%%%%%%

In this section, we give a description of constant mean curvature (cmc) $H$
surfaces which are invariant under translations along a
horizontal geodesic. 

Let us fix a geodesic $\Gamma$ in $\H^2$ and consider $(\phi,\theta)$
polar coordinates at an endpoint of $\Gamma$ such that $\Gamma=
\{\theta=\pi/2\}$. The translations along $\Gamma$ are given by
$\phi\mapsto\phi+constant$. 

Actually, we study cmc graphs which gives a local description of
translation
invariant surfaces; on such a graph, we choose the upward pointing normal.
Let $u$ be a function defined on $\Omega\subset\H^2$, the graph of $u$ has
constant mean curvature $H$ if $u$ satisfies 
\begin{equation}\label{cmc}
\mbox{div}\left(\frac{\nabla u}{W_u}\right)= \mbox{div}\left(X_u\right)=2H
\end{equation}
In the following we assume $H>0$ \textit{i.e.} the mean curvature vector is
upward pointing. Let $u$ be a cmc graph invariant by the translations along
$\Gamma$. Then $u$ can be written as $u(\phi,\theta)=f(\theta)$. We have
$\nabla u =\sin^2(\theta)f'(\theta)\frac{\partial}{\partial\theta}$. Let
$\theta_0,\theta_1\in(0,\pi)$  with $\theta_0<\theta_1$ and
$\phi_0,\phi_1\in\R$ with $\phi_0<\phi_1$. Using \eqref{cmc}, the
Divergence Theorem gives us:
$$
\int_{\partial([\phi_0,\phi_1]\times [\theta_0,\theta_1])} \langle
X_u,\nu\rangle=2H \mathrm{Area}\left([\phi_0,\phi_1]\times
[\theta_0,\theta_1]\right)
$$

Then
$$
\int_{\phi_0}^{\phi_1}\frac{f'(\theta_1)}{\sqrt{
1+\sin^2(\theta_1){f'(\theta_1)}^2}}d\phi-
\int_{\phi_0}^{\phi_1}\frac{f'(\theta_0)}{\sqrt{
1+\sin^2(\theta_0){f'(\theta_0)}^2}}d\phi =
2H\int_{\phi_0}^{\phi_1}\int_{\theta_0}^{\theta_1}\frac{1}{\sin^2(\theta)}
d\theta d\phi
$$

Thus $u$ is a cmc $H$ graph if and only if $f$ satisfies:
$$
\dfrac{d}{d\theta}\left(\frac{f'}{\sqrt{
1+\sin^2\theta\left|f'\right|^2}} \right)=\frac{2H}{\sin^2(\theta)}
$$
Hence $f'$ satisfies:
\begin{equation}\label{integre}
\frac{f'}{\sqrt{ 1+\sin^2\theta\left|f'\right|^2}}
=-2H\cot(\theta)+A
\end{equation}

We notice that changing $\theta$ by $\pi-\theta$ replaces $A$ by $-A$;
thus, in the following we assume $A\ge 0$. 

\underline{Case $H=0$} (Figure \ref{H=0}). We have $\dis
f'=\frac{A}{\sqrt{1-A^2\sin^2(\theta)}}$. Thus there are three subcases:
\begin{enumerate}
\item $A<1$. $f'$ and $f$ are defined on $(0,\pi)$, $u$ is an entire
graph. Moreover $f$ takes finite boundary value at $0$ and $\pi$.
\item $A=1$. $f'$ is defined on $(0,\pi/2)$ by $f'=1/\cos(\theta)$. Then
$f$ is defined on $(0,\pi/2)$ and takes a finite boundary value at $0$ and
diverges to $+\infty$ at $\pi/2$. 
\item $A>1$. $f'$ and $f$ are defined on $(0,\theta_1)$, with $\theta_1=
\arcsin(1/A)$. $f$ takes finite boundary values at $0$ and $\theta_1$ and
$\dis\frac{d f}{d\nu}(\theta_1)=+\infty$.
\end{enumerate}

\begin{figure}[h]
\begin{center}
\resizebox{0.8\linewidth}{!}{\input{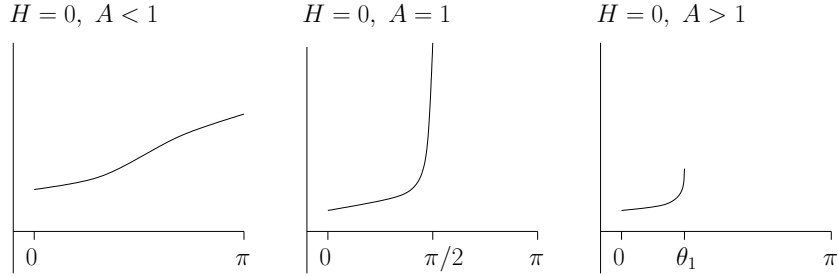}}
\caption{$H=0$ case}\label{H=0}
\end{center}
\end{figure}

Let us now study the case $H>0$. Equation \eqref{integre} can be written:
$$
\frac{\sin(\theta)f'}{\sqrt{ 1+\sin^2\theta\left|f'\right|^2}}
=-2H(\cos(\theta)-k\sin(\theta))
$$
where $2H k=A$ ($k\ge 0$). Then $f'$ is defined when
$|\cos(\theta)-k\sin(\theta)|<1/(2H)$ by
$$
f'(\theta)=\frac{-2Hg(\theta)}{\sin(\theta)\sqrt{1-4H^2g^2(\theta)}}
$$
We define $g(\theta)=\cos(\theta)-k\sin(\theta)$.
$g'(\theta)=-\sin(\theta)-k\cos(\theta)$, thus $g'(\theta)=0$ for
$\theta=\theta_0=\pi+\arctan(-k)$. We have $g(\theta_0)=-\sqrt{1+k^2}$.
The behaviour of $g$ is summarized in the following table.
$$
\begin{array}{|c|ccccr|}
\hline
&0&&\theta_0&&\pi\\
\hline
g'(\theta)&-k&-&0&+&k\\
\hline
&1&&&&-1\\
g&&\searrow &&\nearrow&\\
&&&-\sqrt{1+k^2}&&\\
\hline
\end{array}
$$

A. \underline{Case $H<1/2$} (Figure \ref{Hpetit}). There are three
sub-cases:
\begin{enumerate}
\item[A1.] $k<\sqrt{(1/2H)^2-1}$. $f'$ and $f$ are defined on $(0,\pi)$,
$u$ is an entire graph. $f$ takes boundary value $+\infty$ at $0$ and
$\pi$.
\item[A2.] $k=\sqrt{(1/2H)^2-1}$. $f'$ and $f$ are defined on
$(0,\theta_0)$
and $(\theta_0,\pi)$. $f$ takes boundary value $+\infty$ at $0$ and $\pi$,
$\lim_{{\theta_0}^-} f=+\infty$ and
$\lim_{{\theta_0}^+}f=-\infty$.
\item[A3.] $k>\sqrt{(1/2H)^2-1}$. There are $\theta_1$ and $\theta_2$ with
$0<\theta_1<\theta_0<\theta_2<\pi$ such that $f'$ and $f$ are defined on
$(0,\theta_1)$ and $(\theta_2,\pi)$. $f$ takes finite boundary value
at $\theta_1$ and $\theta_2$, $+\infty$ at $0$ and $\pi$,
$\dis\frac{d f}{d\nu}(\theta_1)=+\infty$ and $\dis\frac{d
f}{d\nu}(\theta_2)=-\infty$.
\end{enumerate}

\begin{figure}[h]
\begin{center}
\resizebox{0.8\linewidth}{!}{\input{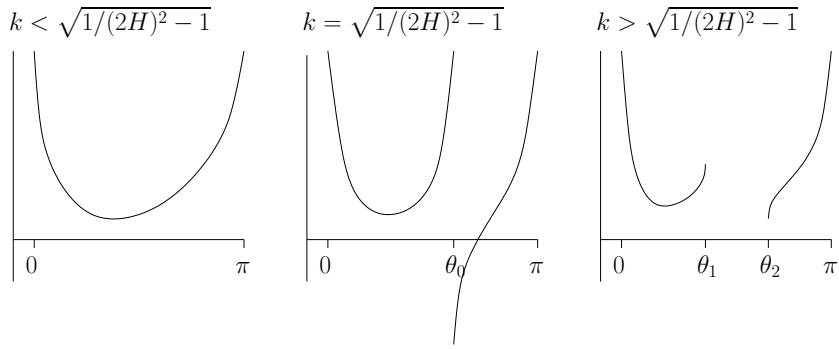}}
\caption{$H<1/2$ case}\label{Hpetit}
\end{center}
\end{figure}

B. \underline{Case $H=1/2$} (Figure \ref{Hgrand}). There are two subcases:
\begin{enumerate}
\item[B1.] $k=0$. $f'$ is defined on $(0,\pi)$ by $f'=\dis
-\frac{\cos(\theta)}{\sin^2(\theta)}$. Hence $f$ is defined on $(0,\pi)$ by
$f=\dis\frac{1}{\sin(\theta)}+K$: $f$ takes boundary value $+\infty$ at
$0$ and $\pi$.
\item[B2.] $k>0$. There is $\theta_1\in(0,\theta_0)$ such that $f'$ and $f$
are defined on $(0,\theta_1)$. $f$ takes finite boundary value at
$\theta_1$, $\dis\frac{d f}{d\nu}(\theta_1)=+\infty$ and boundary value
$+\infty$ at $0$.
\end{enumerate}

C. \underline{Case $H>1/2$} (Figure \ref{Hgrand}). There are $\theta_1$ and
$\theta_2$ with
$0<\theta_1<\theta_2<\theta_0$ such that $f'$ and $f$ are defined on
$(\theta_1,\theta_2)$. $f$ takes finite boundary value at $\theta_1$ and
$\theta_2$, $\dis\frac{d f}{d\nu}(\theta_1)=+\infty$ and $\dis\frac{d
f}{d\nu}(\theta_2)=+\infty$

\begin{figure}[h]
\begin{center}
\resizebox{0.8\linewidth}{!}{\input{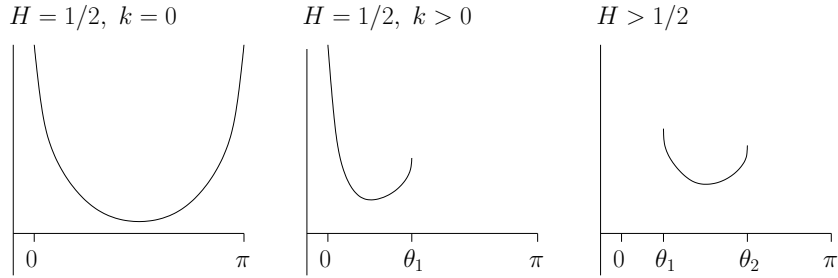}}
\caption{$H=1/2$ and $H>1/2$ cases}\label{Hgrand}
\end{center}
\end{figure}

%%%%%%%%%%%%%%%%%%%
%%%%%%%%%%%%%%%%%%%

\bibliographystyle{plain}
\bibliography{bill}

%%%%%%%%%%%%%%%%%%%
%%%%%%%%%%%%%%%%%%%

\textsc{Laurent Mazet}

Universit\'e Paris-Est,

Laboratoire d'Analyse et Math\'ematiques Appliqu\'ees, UMR 8050

UFR de Sciences et technologies, D\'epartement de Math\'ematiques

61 avenue du g\'en\'eral de Gaulle 94010 Cr\'eteil cedex (France)

\texttt{laurent.mazet@univ-paris12.fr}

\medskip

\textsc{M. Magdalena Rodr\'\i guez}

Universidad Complutense de Madrid,

Departamento de \'Algebra

Plaza de las Ciencias, 3

28040 Madrid (Spain)

\texttt{magdalena@mat.ucm.es}

\medskip

\textsc{Harold Rosenberg}

Universit\'e Denis Diderot Paris 7,

Institut de Math\'ematiques de Jussieu, UMR 7586

UFR de Math\'ematiques, Equipe de G\'eom\'etrie et Dynamique

Site de Chevaleret

75205 Paris cedex 13 (France)

\texttt{rosen@math.jussieu.fr}
\end{document}